\documentclass[12pt]{amsart}
\usepackage{amsmath,amsfonts,euscript,dsfont,amscd,amsthm,amssymb,upref,graphics}
\usepackage[all]{xy}

\theoremstyle{definition}

\swapnumbers

\theoremstyle{plain}

\newtheorem{theorem}{Theorem}[section]
\newtheorem{proposition}[theorem]{Proposition}
\newtheorem{lemma}[theorem]{Lemma}
\newtheorem{corollary}[theorem]{Corollary}

\newtheorem{sub}{}[theorem] 

\theoremstyle{definition}

\newtheorem{definition}[theorem]{Definition}

\newtheorem{parag}[theorem]{}

\newtheorem{subparag}[sub]{}

\newtheorem{remark}[theorem]{Remark}
\newtheorem{remarks}[theorem]{Remarks}

\theoremstyle{remark}

\newtheorem*{smallremark}{Remark}

{\begin{enumerate}\setlength{\itemsep}{#1}}{\end{enumerate}}
\newenvironment{enumerata}%
{\begin{enumerate}

}{\end{enumerate}}
{\begin{enumerata}\setlength{\itemsep}{#1}}{\end{enumerata}}

\newcommand{\Aut}{	\operatorname{{\rm Aut}}}
\newcommand{\Spec}{	\operatorname{{\rm Spec}}}

\newcommand{\depth}{	\operatorname{{\rm depth}}}

\newcommand{\Char}{	\operatorname{{\rm char}}}

\newcommand{\Cl}{	\operatorname{{\rm Cl}}}

\newcommand{\dom}{	\operatorname{{\rm dom}}}
\newcommand{\codom}{	\operatorname{{\rm codom}}}

\newcommand{\id}{	\operatorname{{\rm id}}}
\newcommand{\exc}{	\operatorname{{\rm exc}}}

\newcommand{\Miss}{	\operatorname{{\rm Miss}}}
\newcommand{\Cont}{	\operatorname{{\rm Cont}}}
\newcommand{\cent}{	\operatorname{{\rm cent}}}

\newcommand{\Bir}{	\operatorname{{\rm Bir}}}

\newcommand{\bZ}{\mathbf{Z}}
\newcommand{\bM}{\mathbf{M}}
\newcommand{\bV}{\text{\boldmath $\Veul$}}
\newcommand{\bH}{\text{\boldmath $\Heul$}}
\newcommand{\bG}{\text{\boldmath $\Geul$}}

\newcommand{\setspec}[2]{\big\{\,#1\, \mid \,#2\, \big\}}

\newcommand{\Nat}{\ensuremath{\mathbb{N}}}

\newcommand{\aff}{\ensuremath{\mathbb{A}}}
\newcommand{\proj}{\ensuremath{\mathbb{P}}}
\newcommand{\bk}{{\ensuremath{\rm \bf k}}}

\newcommand{\kk}[1]{\bk^{[#1]}}

\newcommand{\cgoth}{{\ensuremath{\mathfrak{c}}}}

\newcommand{\Cgoth}{{\ensuremath{\mathfrak{C}}}}

\newcommand{\fgoth}{{\ensuremath{\mathfrak{f}}}}

\newcommand{\pgoth}{{\ensuremath{\mathfrak{p}}}}

\newcommand{\Aeul}{\EuScript{A}}

\newcommand{\Ceul}{\EuScript{C}}
\newcommand{\Deul}{\EuScript{D}}
\newcommand{\Eeul}{\EuScript{E}}

\newcommand{\Geul}{\EuScript{G}}
\newcommand{\Heul}{\EuScript{H}}

\newcommand{\Meul}{\EuScript{M}}
\newcommand{\Neul}{\EuScript{N}}
\newcommand{\Oeul}{\EuScript{O}}

\newcommand{\Veul}{\EuScript{V}}

\newcommand{\isom}{\cong}
\renewcommand{\epsilon}{\varepsilon}
\renewcommand{\phi}{\varphi}
\renewcommand{\emptyset}{\varnothing}

\newlength{\mylength}
\newcommand{\BirA}{\Bir(\aff^2)}
\newcommand{\AutA}{\Aut(\aff^2)}

\newcommand{\rien}[1]{}

\CompileMatrices

\begin{document}
\renewcommand{\baselinestretch}{1.07}


\title[Birational endomorphisms of the affine plane]{Compositions of\\ birational endomorphisms of the affine plane}

\author{Pierrette Cassou-Nogu\`es}
\author{Daniel Daigle}

\address{IMB, Universit\'e Bordeaux 1 \\
351 Cours de la lib\'eration, 33405, Talence Cedex, France}
\email{Pierrette.Cassou-nogues@math.u-bordeaux1.fr}

\address{Department of Mathematics and Statistics\\
University of Ottawa\\
Ottawa, Canada\ \ K1N 6N5}
\email{ddaigle@uottawa.ca}

\thanks{Research of the first author partially supported by Spanish grants MTM2010-21740-C02-01 and  MTM2010-21740-C02-02.}
\thanks{Research of the second author supported by grant RGPIN/104976-2010 from NSERC Canada.}

{\renewcommand{\thefootnote}{}
}

{\renewcommand{\thefootnote}{}
}

\begin{abstract} 
Besides contributing several new results in the general theory of birational endomorphisms of $\aff^2$,
this article describes certain classes of birational endomorphisms $f : \aff^2 \to \aff^2$ defined by requiring
that the missing curves or contracting curves of $f$ are lines.
The last part of the article is concerned with the monoid structure of the 
set of birational endomorphisms of $\aff^2$.
\end{abstract}

\maketitle
  
\vfuzz=2pt


Let $\bk$ be an algebraically closed field
and let $\aff^2 = \aff^2_\bk$ be the affine plane over $\bk$.
A {\it birational endomorphism of $\aff^2$}
is a morphism of algebraic varieties, $f:\aff^2 \to \aff^2$,
which restricts to an isomorphism $U\to V$ for some nonempty Zariski-open
subsets $U$ and $V$ of $\aff^2$.
The set $\Bir( \aff^2 )$ of birational endomorphisms of $\aff^2$
is a monoid under composition of morphisms, and the group of invertible elements
of this monoid is the automorphism group $\Aut(\aff^2)$.
An element $f$ of $\Bir( \aff^2 )$ is {\it irreducible\/}
if it is not invertible and if, for every factorization $f = h \circ g$ with $g,h \in \Bir( \aff^2 )$,
one of $g,h$ is invertible.
Elements $f,g \in \Bir( \aff^2 )$ are {\it equivalent} (denoted $f \sim g$) if there exist $u,v \in \Aut( \aff^2 )$
satisfying $u \circ f \circ v = g$.
The elements of $\Bir( \aff^2 )$ which are equivalent to the birational morphism
$c : \aff^2 \to \aff^2$, $c(x,y) = (x,xy)$,
are called  {\it simple affine contractions\/} (SAC), and are the simplest examples of non invertible elements
of $\Bir( \aff^2 )$.
It was at one time an open question whether $\Aut(\aff^2) \cup \{c\}$ generated $\Bir( \aff^2 )$ as a monoid
(the question arose in Abhyankar's seminar at Purdue in the early 70s);
P.\ Russell 
showed that the answer was negative by giving
an example (which appeared later in \cite[4.7]{Dai:bir}) of an irreducible
element of $\Bir( \aff^2 )$ which is not a SAC.
This was the first indication that $\Bir( \aff^2 )$ could be a complicated object.

Papers \cite{Dai:bir} and \cite{Dai:trees} seem to be the first publications
that study birational endomorphisms of $\aff^2$ in a systematic way (these are our main references).
We know of two more contributions to the subject: a certain family of elements of $\BirA$
is described explicitly in \cite{CassouRussell:BirEnd},
and \cite{ShpilYu:BirMor} gives an algorithm for deciding whether a given element of $\Bir( \aff^2 )$ is in 
the submonoid generated by $\Aut(\aff^2) \cup \{c\}$.

The list of references is much longer if we include another aspect of the problem.
Indeed, there is a long history of studying polynomials $F \in \bk[X,Y]$ which appear as components of
birational endomorphisms of $\aff^2$.
To explain this, we recall some definitions.
A polynomial $F \in \bk[X,Y]$ is called a {\it field generator\/} if there exists $G \in \bk(X,Y)$
satisfying $\bk(F,G)=\bk(X,Y)$;
in the special case where $G$ can be chosen in $\bk[X,Y]$, one says that $F$ is a {\it good\/} field generator.
So a good field generator is just the same thing as a component of a birational endomorphism of $\aff^2$.
By a {\it generally rational polynomial\/}\footnote{Generally rational polynomials
are sometimes called ``rational polynomials'' or ``generically rational polynomials''.}
 we mean an $F \in \bk[X,Y]$ such that, for almost all $\lambda \in \bk$,
$F-\lambda$ is an irreducible polynomial whose zero-set in $\aff^2$ is a rational curve
(where ``almost all'' means ``all but possibly finitely many'').
If $\Char\bk=0$ then field generators and generally rational polynomials are one and the same thing
(this is noted in \cite{MiySugie:GenRatPolys}; see \cite{Dai:GenRatPols} for the positive characteristic case).
The study of these polynomials is a classical subject,
as is clear from considering the following (incomplete) list of references:
\cite{Nishino_68},
\cite{Nishino_69},
\cite{Nishino_70},
\cite{Saito_72},
\cite{Saito_77},
\cite{JanThesis},
\cite{Rus:FieldGen},
\cite{Rus:fg2},
\cite{MiySugie:GenRatPolys},
\cite{Kal:TwoRems},
\cite{BartoCassou:RemPolysTwoVars},
\cite{NeumannNorbury:simple},
\cite{Cassou-BadFG},
\cite{Sasao_QuasiSimple2006},
\cite{Dai:GenRatPols},
\cite{CassouDaigVGVBFG}.

\medskip
This paper is a contribution to the theory of birational endomorphisms of $\aff^2$.
Our methods are those of \cite{Dai:bir} and \cite{Dai:trees},
and we place ourselves at the same level of generality as in those papers:
the base field $\bk$ is algebraically closed but otherwise arbitrary.

In \cite{Dai:bir}, \cite{Dai:trees} and \cite{CassouRussell:BirEnd},
there is a tendency to restrict one's attention to irreducible elements of $\BirA$.
Going in an orthogonal direction, the present paper is mainly concerned
with {\it compositions\/} of birational endomorphisms.
This choice is motivated by many reasons.
First, the examples given in \cite{Dai:bir}, \cite{Dai:trees} and \cite{CassouRussell:BirEnd}
show that $\Bir( \aff^2 )$ contains a great diversity of irreducible elements of arbitrarily high complexity;
since the task of finding all irreducible elements is probably hopeless,
it seems to us that finding more examples of them might be less relevant than, say,
trying to understand the monoid structure of $\BirA$.
Also, a significant portion of this paper is geared towards
proving Theorem~\ref{du-239412948ynf}, which we need in the forthcoming \cite{CN-Daigle:Lean} for proving
the following fact:
{\it Let $\bk$ be an arbitrary field and $A_0 \supset A_1 \supset \cdots$ an infinite,
strictly descending sequence of rings such that
\mbox{\rm (i)}~each $A_i$ is a polynomial ring in $2$ variables over $\bk$;
\mbox{\rm (ii)}~all $A_i$ have the same field of fractions;
and \mbox{\rm (iii)}~the ring $R = \bigcap_i A_i$ is not equal to $\bk$;
then $R=\bk[F]$ for some $F$, where $F$ is a good field generator of $A_0$
and a variable of $A_i$ for $i \gg0$.}
Moreover, if one is interested in field generators and generally rational polynomials,
one should not restrict one's attention to irreducible endomorphisms.
In this respect we point out that the components of the morphisms described by 
Theorem~\ref{du-239412948ynf} are precisely the ``rational polynomials of simple type'' listed
in \cite{NeumannNorbury:simple}.

\medskip
The paper is organized as follows.

Section~\ref{Sec:AdmissConfig} contains some preliminary observations on
``admissible'' and ``weakly admissible'' configurations of curves in $\aff^2$.

Section~\ref{SEC:BirEndos} gives several new results in the general theory of birational endomorphisms of $\aff^2$
(in particular \ref{g64f4j3k21ld}, \ref{938hdb67-gdhifa}, \ref{ie123918209ejd10} and \ref{sdkfqp923wr2390094rjdfk},
but also several useful lemmas).

Given $f \in \BirA$, let $\Miss(f)$ (resp.\ $\Cont(f)$) be the set of missing curves (resp.\ contracting curves) 
of $f$; see \ref{dkfjqp923p1`2m3k23} for definitions.
Section~\ref{SECCompositionsofsimpleaffinecontractions} studies
those $f \in \BirA$ satisfying the condition that $\Miss(f)$ is a weakly admissible configuration,
or the stronger condition that $\Miss(f)$ is an admissible configuration,
or the even stronger condition that both $\Miss(f)$ and $\Cont(f)$ are admissible configurations.
The main result of Section~\ref{SECCompositionsofsimpleaffinecontractions},
Theorem~\ref{du-239412948ynf}, gives a complete description of these three subsets of $\BirA$.

While Sections~\ref{SEC:BirEndos} and \ref{SECCompositionsofsimpleaffinecontractions} are mainly concerned with
individual endomorphisms, Section~\ref{SomeaspectsofthemonoidBirA} considers the monoid structure of $\BirA$.
The first part of that section shows,
in particular, that if $S$ is a subset of $\BirA$ such that $\AutA \cup S$ generates
$\BirA$ as a monoid, then $\setspec{ \deg f }{ f \in S }$ is not bounded and $|S|=|\bk|$ 
(giving a very strong negative answer to the already mentioned question of Abhyankar).
The second part shows that the submonoid $\Aeul$ of $\BirA$ generated by SACs and automorphisms
is ``factorially closed'' in $\BirA$, i.e., if $f,g \in \BirA$
are such that $g \circ f \in \Aeul$ then $f,g \in \Aeul$.

\medskip
The second author would like to thank his teacher, Peter Russell,
for introducing him to birational endomorphisms of $\aff^2$.
He would also like to thank the faculty and staff at the Universit\'e de Bordeaux I for their hospitality.
The research leading to this paper was initiated when the second author spent
one month at that institution as a {\it professeur invit\'e}.

\medskip
\noindent{\bf Conventions.}
All algebraic varieties (in particular all curves and surfaces) are irreducible and reduced.
All varieties and morphisms are over an algebraically closed field $\bk$ of arbitrary characteristic
($\bk$ is assumed to be algebraically closed from \ref{weaklyioufq923lksmdjjj} until the end of the paper).
The word ``point'' means ``closed point'', unless otherwise specified.

All rings are commutative and have a unity. The symbol $A^*$ denotes the set of units of a ring $A$.
If $A$ is a subring of a ring $B$ and $n \in \Nat$, the notation $B = A^{[n]}$ means that $B$
is isomorphic (as an $A$-algebra) to the polynomial ring in $n$ variables over $A$. 
We adopt the conventions that $0 \in \Nat$, that ``$\subset$'' means strict inclusion
and that ``$\setminus$'' denotes set difference.

\section{Admissible configurations of curves in $\aff^2$}
\label {Sec:AdmissConfig}

Recall the following terminology.
Let $\bk$ be a field, $A = \kk2$, and $\aff^2_\bk = \Spec A$.
We abbreviate $\aff^2_\bk$ to $\aff^2$.
By a {\it coordinate system of\/} $A$, 
we mean an ordered pair $(F,G)$ of elements of $A$ satisfying $A=\bk[F,G]$.
A {\it variable\/} of $A$ is an element $F \in A$ for which there
exists $G$ satisfying $\bk[F,G]=A$.

Let $F \in A$ be an irreducible element and 
let $C \subset \aff^2$ be the zero-set of $F$
(i.e., the set of prime ideals $\pgoth \in \Spec A = \aff^2$ satisfying $F \in \pgoth$);
we call $C$ a {\it line\/} if $A/FA = \kk1$,
and a {\it coordinate line\/} if $F$ is a variable of $A$.
Note that $C$ is a line iff $C \isom \aff^1$;
given a coordinate system $(X,Y)$ of $A$, $C$ is a coordinate line iff
some automorphism of $\aff^2$ maps $C$ onto the zero-set of $X$.
It is clear that all coordinate lines are lines,
and the Epimorphism Theorem (\cite{A-M:line}, \cite{Suzuki})
states that the converse is true if $\Char\bk=0$.
It is known that not all lines are coordinate lines if $\Char\bk \neq0$
(on the subject of lines which are not coordinate lines,
see e.g.~\cite{Ganong:Survey} for a recent survey).
Coordinate lines are sometimes called rectifiable lines.

\medskip
By a {\it coordinate system of\/} $\aff^2 = \Spec A$, we mean a coordinate system of $A$.
That is, a coordinate system of $\aff^2$ is a pair $(X,Y) \in A \times A$ satisfying $A = \bk[X,Y]$.

\medskip
We adopt the viewpoint that $A$ (or $\aff^2$) does not come equipped with a preferred coordinate system,
i.e., no coordinate system is better than the others.  
This may be confusing to some readers, especially those who like to identify $\aff^2$ with $\bk^2$,
because any such identification inevitably depends on the choice of a coordinate system.
So perhaps the following trivial remarks (\ref{43fg6gxcfzx}) deserve to be made.

\begin{parag}  \label {43fg6gxcfzx}
Let $\Cgoth$ denote the set of coordinate systems of $\aff^2$ (or $A$).
\begin{enumerata}

\item Given $\cgoth=(X,Y) \in \Cgoth$ and an element $F \in A$,
one can consider the map $\bk^2 \to \bk$, $(x,y) \mapsto F(x,y)$,
defined by first writing $F = \sum_{i,j} a_{ij}X^iY^j$ with $a_{ij} \in \bk$ (recall that $A=\bk[X,Y]$)
and then setting $F(x,y) =  \sum_{i,j} a_{ij}x^iy^j$ for $(x,y) \in \bk^2$.
One can then consider the zero-set $\bZ(F) \subseteq \bk^2$ of that map $F$.  We stress that 
the map $(x,y) \mapsto F(x,y)$ and the set  $\bZ(F)$ depend on {\it both $F$ and $\cgoth$};
we should write $F_\cgoth(x,y)$ and $\bZ_\cgoth(F)$, but we omit the $\cgoth$.

\item  \label {931--dj1092hd299}
Let $P,Q \in \bk[T_1,T_2]$,
$P = \sum_{i,j} a_{ij}T_1^iT_2^j$, $Q = \sum_{i,j} b_{ij}T_1^iT_2^j$,
$a_{ij},b_{ij} \in \bk$.
\begin{enumerata}

\item The pair $(P,Q)$ determines the map $\bk^2 \to \bk^2$, $(x,y) \mapsto (P(x,y), Q(x,y))$,
where we define $P(x,y) = \sum_{i,j} a_{ij}x^iy^j$ and $Q(x,y) = \sum_{i,j} b_{ij}x^iy^j$.

\item Choose $\cgoth=(X,Y) \in \Cgoth$. Then $(P,Q,\cgoth)$ determines
 the morphism of schemes $f : \aff^2 \to \aff^2$ defined by stipulating that $f$ corresponds to the $\bk$-homomorphism
$A \to A$ which maps $X$ to $P(X,Y) = \sum_{i,j} a_{ij}X^iY^j$
and $Y$ to $Q(X,Y) = \sum_{i,j} b_{ij}X^iY^j$ ($P(X,Y), Q(X,Y) \in A = \bk[X,Y]$).

\end{enumerata}

\item Suppose that $\cgoth = (X,Y) \in \Cgoth$ has been chosen.
Then it is convenient to define morphisms of schemes $\aff^2 \to \aff^2$ simply by
giving the corresponding polynomial maps $\bk^2 \to \bk^2$
(this will be done repeatedly in Section~\ref{SECCompositionsofsimpleaffinecontractions}).
To do so, we abuse language as follows: given $P,Q \in \bk[T_1,T_2]$, the sentence
$$
\textit{``$f : \aff^2 \to \aff^2$ is defined by $f(x,y) = (P(x,y),Q(x,y))$''} 
$$
means that $f : \aff^2 \to \aff^2$ is the morphism of schemes determined by $(P,Q,\cgoth)$ as in 
remark~\eqref{931--dj1092hd299}.
For instance one can define $f : \aff^2 \to \aff^2$ by $f(x,y) = (x,xy)$, always keeping
in mind that this $f$ depends on $\cgoth$.

\end{enumerata}
\end{parag}

\begin{lemma}  \label {dfjpaejflka}
Let $F,G \in A=\kk2$, where $\bk$ is any field, and suppose that each
of $F,G$ is a variable of $A$.
Consider the ideal $(F,G)$ of $A$ generated by $F$ and $G$.
\begin{enumerata}

\item If $(F,G)=A$ then $G = aF+b$ for some $a,b \in \bk^*$.

\item  If $A/(F,G) = \bk$ then $A = \bk[F,G]$.

\end{enumerata}
\end{lemma}

\begin{proof}
Choose $Y$ such that $A=\bk[F,Y]$ and define $X=F$.
Then $A=\bk[X,Y]$ and we may write $G$ as a polynomial in $X,Y$,
say $G=P(X,Y)$.

Suppose that $(F,G)=A$.
Then $1 \in (F,G) = (X,P(X,Y)) = (X,P(0,Y))$ implies $P(0,Y) \in \bk^*$.
Writing $P(0,Y)=b\in\bk^*$, we obtain that
$G-b= P(X,Y)-P(0,Y)$ is divisible by $X$;
as $G-b$ is irreducible, $G-b=aX=aF$ with $a\in \bk^*$, and (a) is proved.

To prove (b), we first observe that
the fact that $G=P(X,Y)$ is a variable of $\bk[X,Y]$ and
$P(X,Y) \not\in \bk[X]$ implies that $P$ is ``almost monic'' in $Y$, i.e., 
\begin{equation} \label{diufakjdsfllll}
P(X,Y) = aY^d + p_1(X)Y^{d-1} + \cdots + p_d(X)
\end{equation}
with $d \ge1$, $a \in \bk^*$ and $p_i(X) \in \bk[X]$ for $i=1,\dots, d$.
Now 
$$
\bk=A/(F,G)=\bk[X,Y]/(X,P(X,Y)) = \bk[X,Y]/(X,P(0,Y))
$$
implies that $\deg P(0,Y)=1$, so $d=1$ in \eqref{diufakjdsfllll}.
Then $G = aY+p_1(X)$ and $\bk[F,G] = \bk[X,aY+p_1(X)]=\bk[X,Y]=A$.
\end{proof}

From now-on, and until the end of this paper, we assume that 
$\bk$ is an algebraically closed field of arbitrary characteristic.
Consider $\aff^2 = \aff^2_\bk$.

\begin{definition}  \label {weaklyioufq923lksmdjjj}
Let $C_1, \dots, C_n$ ($n\ge0$) be distinct curves in $\aff^2$, and consider the set
$S = \{ C_1, \dots, C_n \}$. 
We say that $S$ is a {\it weakly admissible configuration\/} if 
\begin{enumerata}

\item each $C_i$ is a coordinate line;

\item for every choice of $i \neq j$ such that $C_i \cap C_j \neq \emptyset$,
$C_i \cap C_j$ is one point and the local intersection number
of $C_i$ and $C_j$ at that point is equal to $1$.

\end{enumerata}
\end{definition}

\begin{lemma}  \label {weak-dpfq32748rjkdkd}
Given distinct curves $C_1, \dots, C_n$ ($n\ge0$) in $\aff^2$,
the following are equivalent:
\begin{enumerata}

\item $\{ C_1, \dots, C_n \}$ is a weakly admissible configuration;

\item there exists a coordinate system of $\aff^2$ with respect to which 
all $C_i$ have degree~$1$.

\end{enumerata}
\end{lemma}

\begin{proof}
We show that (a) implies (b), the converse being trivial.
Suppose that (a) holds.
Write $\aff^2 = \Spec A$ where $A = \kk2$.
We may assume that $n\ge2$, otherwise the assertion is trivial.
Let $F_1, \dots, F_n \in A$ be variables of $A$ whose zero-sets
are $C_1, \dots, C_n$ respectively.
Condition (a) implies that, whenever $i \neq j$, we have either $(F_i,F_j)=A$ or $A/(F_i,F_j)=\bk$.
Consider the graph $G$ whose vertex-set is $\{ F_1, \dots, F_n \}$
and in which distinct vertices $F_i,F_j$ are joined by an edge
if and only if $A/(F_i,F_j)=\bk$.

In the case where $G$ is discrete, \ref{dfjpaejflka} implies
that $F_i=a_iF_1+b_i$, $a_i,b_i \in \bk^*$, for $i=2,\dots,n$.
Let $X=F_1$ and let $Y$ be such that $A=\bk[X,Y]$.
Then all $F_i$ have degree $1$ with respect to the coordinate system $(X,Y)$.

From now-on, assume that $G$ is not discrete.
Then we may assume that $F_1, F_2$ are joined by an edge.
By \ref{dfjpaejflka}, $\bk[F_1,F_2]=A$.
Let $X_1=F_1$ and $X_2=F_2$, then $A=\bk[X_1,X_2]$ and for
each $i \in \{3, \dots, n \}$ we have:
\begin{itemize}

\item if $F_i, F_1$ are not joined by an edge
then \ref{dfjpaejflka} implies that $F_i = a_i X_1 + b_i$ 
for some $a_i,b_i \in \bk^*$, so $F_i$ has degree $1$ with respect to $(X_1,X_2)$;

\item if $F_i, F_2$ are not joined by an edge then $F_i = a_i X_2 + b_i$
for some $a_i,b_i \in \bk^*$, so $F_i$ has degree $1$ with respect to $(X_1,X_2)$;

\item if $F_i$ is linked to each of $F_1, F_2$ by edges, then $\bk[F_i,F_1] = A = \bk[F_i,F_2]$,
so $F_2 = aF_1+\beta(F_i)$ for some $a \in \bk^*$ and $\beta(T) \in \bk[T]$;
then $\beta( F_i ) = X_2 - aX_1$ has degree $1$ with respect to $(X_1,X_2)$ and consequently
$F_i$ has degree $1$.

\end{itemize}
So all $F_i$ have degree $1$ with respect to the coordinate system $(X_1,X_2)$.
\end{proof}

\begin{parag}  \label{difupawejk}
Let $C_1, \dots, C_n$ ($n \ge0$) be distinct curves in a
nonsingular surface $W$.
We say that $\sum_{i=1}^n C_i$ is an {\it SNC-divisor\/} of $W$ if:
\begin{itemize}

\item each $C_i$ is a nonsingular curve;

\item for every choice of $i \neq j$ such that $C_i \cap C_j \neq \emptyset$,
$C_i \cap C_j$ is one point and the local intersection number
of $C_i$ and $C_j$ at that point is equal to $1$;

\item if $i,j,k$ are distinct then $C_i \cap C_j \cap C_k = \emptyset$.

\end{itemize}
If $D=\sum_{i=1}^n C_i$ is an SNC-divisor of $W$
we write $\Geul(D,W)$ for the graph whose vertex set is 
$\{ C_1, \dots, C_n \}$ 
and in which distinct vertices $C_i$, $C_j$ are joined by an edge if
and only if $C_i \cap C_j \neq \emptyset$.
\end{parag}

\begin{definition}  \label {ioufq923lksmdjjj}
Let $C_1, \dots, C_n$ ($n\ge0$) be distinct curves in $\aff^2$.
We say that the set $\{ C_1, \dots, C_n \}$ is an
{\it admissible configuration\/} if 
\begin{enumerata}

\item each $C_i$ is a coordinate line;

\item $D = \sum_{i=1}^n C_i$ is an SNC-divisor of $\aff^2$;

\item the graph $\Geul(D,\aff^2)$ defined in \ref{difupawejk}
is a forest.

\end{enumerata}
\end{definition}

\begin{proposition}  \label {dpfq32748rjkdkd}
Given distinct curves $C_1, \dots, C_n$ ($n\ge0$) in $\aff^2$,
the following are equivalent:
\begin{enumerata}

\item $\{ C_1, \dots, C_n \}$ is an admissible configuration;

\item there exists a coordinate system $(X,Y)$ of $\aff^2$ 
such that $\bigcup_{i=1}^n C_i$ is the zero-set of $\phi(X)Y^j$
for some $j \in \{0,1\}$ and some 
$\phi(X) \in \bk[X]\setminus\{0\}$.

\end{enumerata}
\end{proposition}

\begin{proof}
It's enough to show that (a) implies (b), as the converse is trivial.
Assume that (a) holds.
By \ref{weak-dpfq32748rjkdkd}, we may choose a coordinate system which respect to which 
all $C_i$ have degree $1$.
As $D = \sum_{i=1}^n C_i$ is an SNC-divisor and $\Geul(D,\aff^2)$ is a forest,
$\bigcup_{i=1}^n C_i$ must be one of the following:
\begin{itemize}

\item a union of $n$ parallel lines;

\item a union of $n-1$ parallel lines with another line.

\end{itemize}
Indeed, any other configuration of lines would either contain three concurrent lines
or produce a circuit in the graph.  Now it is clear that (b) is satisfied.
\end{proof}

\section{Birational morphisms $f : X \to Y$ of nonsingular surfaces
\\ with special emphasis on the case $X=Y=\aff^2$}
\label{SEC:BirEndos}

Throughout this section, $\bk$ is an algebraically closed field of arbitrary characteristic
and we abbreviate $\aff^2_\bk$ to~$\aff^2$.
We consider birational morphisms $f : X \to Y$, where $X$ and $Y$ are nonsingular algebraic
surfaces over $\bk$ (a morphism $f : X \to Y$ is {\it birational\/} if there
exist Zariski-open subsets $\emptyset \neq U \subseteq X$ and $\emptyset \neq V \subseteq Y$
such that $f$ restricts to an isomorphism $U \to V$).
We are particularly interested in the case $X = \aff^2 = Y$.

Essentially all the material given in \ref{idjfawe8738324}--\ref{788hkckdssljkhjh25}
 can be found in \cite{Dai:bir}.
From \ref{g64f4j3k21ld} to the end of the section, the material appears to be new
(except \ref{938hdb67-gdhifa}\eqref{000004}).

\begin{parag}  \label {idjfawe8738324}
Let $f: X \to Y$ and $f': X' \to Y'$ be birational morphisms of nonsingular surfaces.
We say that $f,f'$ are {\it equivalent\/} ($f \sim f'$) if
there exist isomorphisms $x : X \to X'$ and 
$y : Y \to Y'$ such that $y \circ f = f' \circ x$.
\end{parag}

\begin{parag}[{\cite[1.2]{Dai:bir}}]  \label {dkfjqp923p1`2m3k23}
Let $f : X \to Y$ be a birational morphism of nonsingular surfaces.
A {\it missing curve\/} of $f$ is a curve $C \subset Y$ whose intersection
with the image of $f$ is a finite set of points.
We write $\Miss(f)$ for the set of missing curves of $f$,
$q(f)$ for the cardinality of $\Miss(f)$ and 
$q_0(f)$ for the number of missing curves of $f$ which are disjoint from $f(X)$.
A {\it contracting curve\/} of $f$ is a curve $C \subset X$ such that
$f(C)$ is a point.
The set of contracting curves of $f$ is denoted $\Cont(f)$,
and $c(f)$ denotes the cardinality of $\Cont(f)$.
The natural numbers $q(f)$, $q_0(f)$ and $c(f)$ are 
invariant with respect to equivalence (\ref{idjfawe8738324}) of birational morphisms,
i.e., $f \sim f' \Rightarrow c(f)=c(f')$ and similarly for $q$ and $q_0$.
Call a point of $Y$ a {\it fundamental point\/} of $f$
if it is $f(C)$ for some contracting curve $C$ of $f$.
The set of fundamental points of $f$ is also called the
{\it center\/} of $f$, denoted $\cent(f)$.
The {\it exceptional locus\/} of $f$ is defined to be
$\exc(f) = f^{-1}\big( \cent(f) \big)$, and is equal to
the union of the contracting curves of $f$.
\end{parag}

\begin{parag}[{\cite[1.1 and 1.2]{Dai:bir}}]  \label {diofj;askdjf}
Let $f : X \to Y$ be a birational morphism of nonsingular surfaces.
There exists a commutative diagram
\begin{equation}  \label {mindec}
\xymatrix{
Y_n \ar[r]^{\pi_n} & \cdots \ar[r]^{\pi_1} & Y_0\\
X \ar @{^{(}->}[u] \ar[rr]_{ f }  &&  Y \ar @{=}[u] 
}
\end{equation}
where ``$\hookrightarrow$'' denotes an open immersion and, for each $i$,
$\pi_i : Y_i \to Y_{i-1}$ is the blowing-up of the nonsingular surface $Y_{i-1}$
at a point $P_i \in Y_{i-1}$.

Define $n(f)$ to be the least natural number $n$ for which there exists
a diagram \eqref{mindec}.
Note that $n(f)$ is invariant with respect to equivalence of birational morphisms.

For each $i \in \{ 1, \dots, n \}$, consider the exceptional curve
$E_i = \pi_i^{-1}(P_i) \subset Y_i$, and let the same symbol $E_i$
also denote the strict transform of $E_i$ in $Y_n$.
It is clear that the union of the contracting curves of $f$ is
the intersection of $E_1 \cup \dots \cup E_n \subset Y_n$
with the open subset $X$ of $Y_n$; thus:
\begin{equation}  \label{difwkejlkaskkkkkk}
c(f) \le n(f),
\end{equation}
\begin{equation}   \label {23748rhf63t}
\begin{minipage}[t]{.9\textwidth} \it
each contracting curve is nonsingular and rational,
$D = \sum_{C \in \Cont(f)} \! C$\ \ is
an SNC-divisor of $X$ and the graph $\Geul(D,X)$ is a forest.\footnotemark
\end{minipage}
\end{equation}
\footnotetext{Note that contracting curves are not necessarily isomorphic to $\aff^1$. 
So, in the case $X = \aff^2 = Y$, $\Cont(f)$ is not necessarily an admissible configuration
in the sense of \ref{ioufq923lksmdjjj}.}
Given $i \in \{1,\dots,n\}$, consider the complete curve $E_i \subset Y_n$.
Note that if $S$ is a projective nonsingular surface
and $\mu : Y_n \hookrightarrow S$ is an open immersion, 
the self-intersection number of $\mu(E_i)$ in $S$ is independent of the 
choice of $(S,\mu)$; we denote this number by $(E_i^2)_{Y_n}$.
Then the following holds (cf.\ \cite[1.2(c)]{Dai:bir}):
\begin{equation}  \label {d023876g77f}
\begin{minipage}[t]{.8\textwidth}
\it
Diagram \eqref{mindec} satisfies $n=n(f)$ if and only if
$(E_i^2)_{Y_n} \le -2$ holds for all $i \in \{1, \dots, n\}$
such that $E_i \subseteq Y_n \setminus X$.
\end{minipage}
\end{equation}
In particular, if Diagram \eqref{mindec} satisfies $n=n(f)$ then:
\begin{equation} \label {1329jf9m39c37hf}
\cent(f) = \setspec{ (\pi_{1} \circ \cdots \circ \pi_{i-1})(P_i) }{ 1 \le i \le n} .
\end{equation}
\end{parag}

The following remarks are trivial consequences of \ref{diofj;askdjf}, but are very useful:

\begin{remarks} \label {9238ch928nhf8n34r}
Let $f : X \to Y$ be a birational morphism of nonsingular surfaces.
\begin{enumerata}

\item \label {pfq903012cm0mjf-i} 
For each $C \in \Miss(f)$, we have $C \cap f(X) \subseteq \cent(f)$.
In particular, the condition $q_0(f)=0$ is equivalent to ``every missing curve contains a fundamental point''.

\item \label {023c093mrfh9129jf-ii} 
Let $C \subset Y$ be a curve.
Then there exists {at most one} curve $C' \subset X$ such that $f(C')$ is a dense subset of $C$.
Moreover, $C'$ exists if and only if $C \notin \Miss(f)$.

\end{enumerata}
\end{remarks}

\begin{lemma}  \label {1dijfaksdjfall00100}
If $\aff^2 \xrightarrow{f}\aff^2 \xrightarrow{g}\aff^2$
are birational morphisms then $n(g\circ f) = n(g) + n(f)$.
\end{lemma}

\begin{proof}
Follows from \cite[2.12]{Dai:bir}.
\end{proof}

\begin{lemma}  \label {dijfaksdjfall00000}
Let $f : \aff^2 \to \aff^2$ be a birational morphism.
\begin{enumerata}

\item  \label{000001}
$q(f) = c(f) \le n(f)$

\item  \label {000002}
$f$ is an isomorphism $\iff n(f)=0 \iff c(f)=0 \iff q(f)=0$.

\item  \label{000003}
Each missing curve of $f$ is rational with one place at infinity.

\item  \label{POdsifias2874wke}
Each fundamental point belongs to some missing curve; 
each missing curve contains some fundamental point.

\item  \label{000006}
If a point $P$ is a singular point of some missing curve of $f$,
or a common point of two missing curves, then $P$ is a fundamental
point of $f$.

\end{enumerata}
\end{lemma}

\begin{proof}
Equality $q(f)=c(f)$ in \eqref{000001} follows from \cite[4.3(a)]{Dai:bir},
and $c(f)\le n(f)$ was noted in \eqref{difwkejlkaskkkkkk}.
Assertion~\eqref{000002} follows from the observation that 
if $n(f)=0$ or $c(f)=0$ then $f$ is an open immersion
$\aff^2 \hookrightarrow \aff^2$ and hence an automorphism.
Assertion~\eqref{000003} follows from result 4.3(c) of \cite{Dai:bir}.
The first (resp.\ the second) assertion of \eqref{POdsifias2874wke} follows
from \cite[2.1]{Dai:bir} 
(resp.\ from the claim that $q_0(f)=0$, in \cite[4.3(a)]{Dai:bir}).
Refer to \cite[1.6]{Dai:tri2} for a proof of assertion~\eqref{000006}.
\end{proof}

Several of the above facts are stated in greater generality in \cite{Dai:bir}.
For instance, if $X \xrightarrow{f} Y \xrightarrow{g} Z$ are birational morphisms of nonsingular surfaces
and $q_0(f)=0$, then (by \cite[1.3]{Dai:bir}) $n(g \circ f) = n(g) + n(f)$.
Also, if $X,Y$ are nonsingular surfaces satisfying $\Gamma(X,\Oeul_X)^*=\bk^*$ and $\Cl(Y)=0$
then (by \cite[2.11]{Dai:bir}) every birational morphism $f : X \to Y$ satisfies $q_0(f)=0$.
The following generalization of \ref{dijfaksdjfall00000}\eqref{000001} also deserves to be noted:

\begin{lemma}  \label {iuqr2378632r6}
Let $f : X \to X$ be a birational morphism, where $X$ is a nonsingular surface.
Then $c(f)=q(f)$.
\end{lemma}

\begin{proof}
Follows from Remark 2.13 of \cite{Dai:bir}.
\end{proof}

\begin{lemma}  \label {788hkckdssljkhjh25}
Let $f : X \to Y$ be a birational morphism of nonsingular surfaces
and $\Gamma_f$  the union of the missing curves of $f$.
If $X$ is affine then the following hold:
\begin{enumerata}

\item $\cent(f) \subseteq \Gamma_f$;
\item $Y \setminus \Gamma_f$ is the interior of $f(X)$ and $f^{-1}\big( Y \setminus \Gamma_f \big) = X \setminus \exc(f)$;
\item the surfaces $X \setminus \exc(f)$ and  $Y \setminus \Gamma_f$ are affine,
and $f$ restricts to an isomorphism $X \setminus \exc(f) \to Y \setminus \Gamma_f$.

\end{enumerata}
\end{lemma}

\begin{proof}
Follows from \cite[Prop.\ 2.1]{Dai:bir} and its proof.
\end{proof}

\begin{proposition}  \label {g64f4j3k21ld}
Let $f : \aff^2 \to \aff^2$ be a birational morphism.
If $P$ is a singular point of a missing curve of $f$,
then $P$ belongs to at least two missing curves of $f$.
\end{proposition}

\begin{proof}
By \ref{dijfaksdjfall00000}\eqref{000006}, $P$ is a fundamental point of $f$;
so it suffices to show that if a fundamental point $P$ belongs to only one
missing curve $C$, then the multiplicity $\mu(P,C)$ of $C$ at $P$ is equal to $1$.
So assume that $P$ is a fundamental point which belongs to only one missing curve $C$.
Choose a diagram \eqref{mindec} satisfying $n=n(f)$, and let the notation
$P_i, E_i$, etc, be as in \ref{diofj;askdjf}.
In fact let us choose diagram \eqref{mindec} in such a way that $P_1=P$ and,
for some $s \in \{ 1, \dots, n \}$,
\begin{equation}  \label{87gf6h7923k4jhkjsdf9a87}
\text{$P_2, \dots, P_s$ are infinitely near $P_1$, but $P_{s+1}, \dots, P_n$ are not.}
\end{equation}
Let us label the missing curves as $C_1, \dots, C_q$, where 
\begin{equation}  \label{43j5;2k32k5jy2iq4}
P_1 \in C_j \Leftrightarrow j=1 .
\end{equation}
The diagram~\eqref{mindec} together with the ordering 
$(C_1, \dots, C_q)$ of the set of missing curves constitutes a ``minimal decomposition'' of $f$,
in the terminology of \cite[1.2(h)]{Dai:bir}.
This minimal decomposition $\Deul$ determines matrices $\mu_\Deul$, $\Eeul_\Deul$,
$\epsilon_\Deul$ and $\epsilon'_\Deul$, defined in \cite[2.8]{Dai:bir}.
These are matrices with entries in $\Nat$, and result \cite[4.3(b)]{Dai:bir} asserts that
the product $\epsilon'_\Deul \mu_\Deul$ is a square matrix of determinant $\pm1$.
We shall now argue that the condition $\det( \epsilon'_\Deul \mu_\Deul ) = \pm1$
implies that $\mu(P_1,C_1)=1$ (this will complete the proof of the proposition).
By \eqref{87gf6h7923k4jhkjsdf9a87}, the $n\times n$ matrix $\Eeul_\Deul$ has the following shape:
$$
\Eeul_\Deul = (e_{ij}) =  \left( \begin{array}{c|c} \Eeul_0 & 0 \\ \hline 0 & \Eeul_1 \end{array} \right)
\qquad \text{(with $e_{ij} \in \Nat$ for all $i,j$)}
$$
where $\Eeul_0$ is an $s \times s$ lower-triangular matrix with zero diagonal, and where
\begin{equation*}  \label{f8g73k4jsfd76}
\text{the first row is the only zero row of $\Eeul_0$.}
\end{equation*}
Consider the $n\times n$ matrix $\epsilon_\Deul$,
determined by $\Eeul_\Deul$ as explained in \cite[2.7]{Dai:bir}.
The already mentioned properties of $\Eeul_\Deul$ imply that $\epsilon_\Deul$ is as follows:
$$
\epsilon_\Deul = (\epsilon_{ij})
= \left( \begin{array}{c|c} \epsilon_0 & 0 \\ \hline 0 & \epsilon_1 \end{array} \right)
\qquad \text{(with $\epsilon_{ij} \in \Nat$ for all $i,j$)}
$$
where $\epsilon_0$ is an $s \times s$ lower-triangular matrix with diagonal entries equal to $1$,
and where
\begin{equation*}  \label{23h24897dk32}
\text{all entries in the first column of $\epsilon_0$ are positive.}
\end{equation*}

Next, $\epsilon'_\Deul$ is the submatrix of $\epsilon_\Deul$ obtained by deleting the $i$-th row
for each $i \in J$, where
$J=\setspec{ i }{ 1 \le i \le n,\ \, E_i \cap X = \emptyset \text{ in $Y_n$} }$ in the notation of \ref{diofj;askdjf}
($J$ is defined in \cite[1.2(h)]{Dai:bir}).
Let $j_0 = |J \cap \{ 1, \dots, s\}|$;
then the $(n-|J|)\times n$ matrix $\epsilon'_\Deul$ has the form
$$
\epsilon'_\Deul = \left( \begin{array}{c|c} \epsilon'_0 & 0 \\ \hline 0 & \epsilon'_1 \end{array} \right),
$$
where $\epsilon'_0$ is an $(s-j_0) \times s$ matrix with entries in $\Nat$ and
\begin{equation}  \label{dfkjq234jh787dfh}
\text{all entries in the first column of $\epsilon'_0$ are positive.}
\end{equation}
Finally, consider the $n\times q$ matrix $\mu_\Deul$;
by \eqref{43j5;2k32k5jy2iq4},
$$
\mu_\Deul = \left( \begin{array}{c|c} F & 0 \\ \hline G & H \end{array} \right)
\qquad \text{where}\ \ 
F = \mbox{\scriptsize $\left( \begin{array}{c} \mu(P_1,C_1) \\ \vdots \\ \mu(P_s,C_1) \end{array} \right)$}
\text{\ is $s\times 1$.}
$$
We have
\begin{equation}  \label{sd9f0g8;32lk5jsd908kj324}
\epsilon'_\Deul \mu_\Deul
= \left( \begin{array}{c|c} \epsilon'_0 & 0 \\ \hline 0 & \epsilon'_1 \end{array} \right)
\left( \begin{array}{c|c} F & 0 \\ \hline G & H \end{array} \right)
= \left( \begin{array}{c|c} \epsilon'_0 F & 0 \\ \hline \epsilon'_1 G  & \epsilon'_1 H \end{array} \right) ,
\end{equation}
where the block $\epsilon'_0 F$ has size $(s-j_0) \times 1$.
By \eqref{87gf6h7923k4jhkjsdf9a87}, $(E_s^2)_{Y_n} = -1$;
so $E_s \nsubseteq Y_n \setminus X$ by \eqref{d023876g77f} and hence $s \notin J$ by definition of $J$.
It follows that $s-j_0\ge1$.
In view of \eqref{sd9f0g8;32lk5jsd908kj324}, the fact that $\det( \epsilon'_\Deul \mu_\Deul ) = \pm1$ 
implies that $s-j_0=1$ and that the unique entry of $\epsilon'_0 F$ is $\pm1$.
We have $\{ 1, \dots, s \} \setminus J = \{s\}$, so
$\epsilon'_0 = ( \epsilon_{s1}\, \dots\, \epsilon_{ss} )$
and $\sum_{j=1}^s \epsilon_{sj} \mu(P_j,C_1) = \pm1$.
Since $\epsilon_{sj} \in \Nat$ for all $j$  and (by \eqref{dfkjq234jh787dfh})  $\epsilon_{s1}\ge1$,
we get $1 \le \mu(P_1,C_1) \le \sum_{j=1}^s \epsilon_{sj} \mu(P_j,C_1) = \pm1$,
so $\mu(P_1,C_1)=1$.  This completes the proof.
\end{proof}

\begin{remark}
Let $\phi : X \to Y$ be a dominant morphism of nonsingular surfaces.
By a {\it deficient curve\/} of $\phi$, we mean a curve $C \subset Y$ satisfying:
$$
\text{for almost all points $P \in C$, $| f^{-1}(P) | < [ \bk(X):\bk(Y) ]_s$}
$$
where ``almost all'' means ``all except possibly finitely many,'' ``$|\ |$'' denotes cardinality,
$\bk(X)$ and $\bk(Y)$ are the function fields of $X$ and $Y$ and 
$[ \bk(X):\bk(Y) ]_s$ is the separable degree of the field extension $\bk(X)/\bk(Y)$.
Note that $\phi$ has finitely many deficient curves, and that 
if $\phi$ is birational then the deficient curves are precisely the missing curves.

Then it is interesting to note that Proposition~\ref{g64f4j3k21ld} does not generalize to all dominant morphisms
$\aff^2 \to \aff^2$,
i.e., {\it it is not the case that each singular point of a deficient curve is a common point of at least two
deficient curves.}  This is shown by the following example, in which we assume that $\Char\bk=0$.

Choose a coordinate system of $\aff^2$ and define morphisms
$\aff^2 \xrightarrow{ f }\aff^2 \xrightarrow{ g }\aff^2 \xrightarrow{ h }\aff^2$ by:
$$
f(x,y) = (x,xy), \qquad
g(x,y) = (x+y(y^2-1), y), \qquad
h(x,y) = (x,y^2).
$$
Note that $f$ is a SAC and $g \in \Aut(\aff^2)$. Define $\phi = h \circ g \circ f : \aff^2 \to \aff^2$.
Then $\phi$ has two deficient curves $C_1$ and $C_2$, where:
\begin{itemize}
\setlength{\itemsep}{.5mm}

\item $C_1$ is ``\,$y=0$\,''  (the deficient curve of $h$);

\item $C_2$ is ``\,$x^2-y(y-1)^2=0$\,''  (the image by $h \circ g$ of the missing curve of $f$).

\end{itemize}
Moreover, $(0,1)$ is a singular point of $C_2$ which is not on $C_1$.
\end{remark}

\begin{lemma}  \label {89x21bxnrmv7tr36}
Let $\aff^2 \xrightarrow{f} \aff^2 \xrightarrow{g} \aff^2$ be birational morphisms.
Then 
$$
\cent(g \circ f) = \cent(g) \cup g\big( \cent(f) \big).
$$
In particular, every fundamental point of $g$ is a fundamental point of $g \circ f$.
\end{lemma}

\begin{proof}
Let $X \xrightarrow{f} Y \xrightarrow{g} Z$ be birational morphisms of nonsingular surfaces.
The reader may easily verify that 
$\cent(g \circ f) \subseteq \cent(g) \cup g\big( \cent(f) \big)$ and 
$g\big( \cent(f) \big) \subseteq \cent(g \circ f)$.
In order to obtain the desired equality, there remains to show that
\begin{equation}  \label {892cncc8gvex9focmffrtgb}
\cent(g) \subseteq \cent(g \circ f).
\end{equation}
We claim that \eqref{892cncc8gvex9focmffrtgb} is true whenever $q_0(f)=0$.
Indeed, consider $P \in \cent(g)$. Then there exists a curve $C \subset Y$ such that $g(C)=\{P\}$.
If $C \notin \Miss(f)$ then (\ref{9238ch928nhf8n34r}\eqref{023c093mrfh9129jf-ii})
there exists a curve $C' \subset X$ such that $f(C')$ is a dense subset of $C$;
in particular, $(g \circ f)(C') = \{P\}$ and hence $P \in \cent(g \circ f)$.
If $C \in \Miss(f)$ then, since $q_0(f) = 0$,  \ref{9238ch928nhf8n34r}\eqref{pfq903012cm0mjf-i} implies that
some fundamental point $Q$ of $f$ lies on $C$;
then there exists a curve $C' \subset X$ such that $f(C') = \{Q\}$; then
$(g \circ f)(C') = \{P\}$ and hence $P \in \cent(g \circ f)$.

By \cite[2.11]{Dai:bir}, the condition $q_0(f)=0$ is satisfied whenever $\Gamma(X,\Oeul_X)^* = \bk^*$ and $\Cl(Y)=0$.
In particular, if $X=\aff^2=Y$ then $q_0(f)=0$, so \eqref{892cncc8gvex9focmffrtgb} holds
and consequently $\cent(g \circ f) = \cent(g) \cup g\big( \cent(f) \big)$.
\end{proof}

\begin{lemma}  \label {Doifqrqkkdsjf8}
Let $X \xrightarrow{f} Y \xrightarrow{g} Z$
be birational morphisms of nonsingular surfaces and let
$\Gamma_f$ (resp.\ $\Gamma_g$,  $\Gamma_{g\circ f}$) 
be the union of all missing curves of $f$ (resp.\ of $g$, $g\circ f$).
\begin{enumerata}

\item \label {qo8723yrd187ryqaj}
$\Gamma_{g\circ f}$ is equal to the union of all $1$-dimensional components of
$\Gamma_g \cup \overline{g ( \Gamma_f  )}$,
where $\overline{g ( \Gamma_f  )}$ denotes the closure of $g ( \Gamma_f  )$ in $Z$.

\item \label {qhv6ghqwka7} If $Y$ is affine then  $\Gamma_{g\circ f} = \Gamma_g \cup \overline{g ( \Gamma_f  )}$;
in particular, each missing curve of $f$ is included in $g^{-1}(\Gamma_{g\circ f})$.

\end{enumerata}
\end{lemma}

\begin{proof}
To prove (a),
it's enough to show that a curve in $Z$ is
not included in $\Gamma_{g\circ f}$ if and only if it is not included in
$\Gamma_g \cup \overline{g\big( \Gamma_f \big)}$.
Let $C \subset Z$ be a curve such that $C \nsubseteq \Gamma_{g\circ f}$.
Then there exists a curve $C_0 \subset X$
such that $g(f(C_0))$ is a dense subset of $C$; consequently, the set $C_1 = \overline{f(C_0)}$
is a curve in $Y$ and $g(C_1)$ is dense in $C$, so $C$ is not a missing curve of $g$ and hence
$C \nsubseteq \Gamma_g$. 
If $C \subseteq \overline{g\big( \Gamma_f \big)}$
then there exists a missing curve $C_1'$ of $f$ such that $\overline{g(C_1')}=C$;
however, $C_1$ is the only curve in $Y$ whose image by $g$ is a dense subset of $C$,
and $C_1$ is not a missing curve of $f$; so 
$C \nsubseteq \overline{g\big( \Gamma_f \big)}$ and hence 
$C \nsubseteq \Gamma_g \cup \overline{g\big( \Gamma_f \big)}$.

Conversely, let $C \subset Z$ be a curve such that
$C \nsubseteq \Gamma_g \cup \overline{g\big( \Gamma_f \big)}$.
Then $C \nsubseteq \Gamma_g$, so there exists a curve $C_1 \subset Y$
such that $g(C_1)$ is a dense subset of $C$.
Note that $C_1$ is not a missing curve of $f$, because 
$C \nsubseteq \overline{g\big( \Gamma_f \big)}$;
so there exists a curve $C_0 \subset X$
such that $f(C_0)$ is a dense subset of $C_1$.
Then $(g\circ f)(C_0)$ is a dense subset of $C$ and consequently
$C \nsubseteq \Gamma_{g\circ f}$. This proves (a).

(b) Suppose that $Y$ is affine.  If a point $P \in Z$ is
an irreducible component of $\Gamma_g \cup \overline{g\big( \Gamma_f \big)}$ then
$\{ P \} = g( C )$ where $C$ is a missing curve of $f$,
so $P$ is a fundamental point of $g$; since $Y$ is affine, \ref{788hkckdssljkhjh25}
implies that $\cent(g) \subseteq \Gamma_g$, so $P \in \Gamma_g$,
which contradicts the hypothesis that 
$P$ is an irreducible component of $\Gamma_g \cup \overline{g\big( \Gamma_f \big)}$.
This shows that $\Gamma_g \cup \overline{g\big( \Gamma_f \big)}$ is a union of curves,
so $\Gamma_{g\circ f} = \Gamma_g \cup \overline{g ( \Gamma_f  )}$ follows from~(a).
\end{proof}

Results \ref{938hdb67-gdhifa} and \ref{ie123918209ejd10} are valid in all characteristics,
but are particularly interesting when $\Char\bk>0$
(recall that not all lines are coordinate lines when $\Char\bk > 0$).

\begin{proposition}  \label {938hdb67-gdhifa}
Let $f : \aff^2 \to \aff^2$ be a birational morphism.
\begin{enumerata}

\item  \label {000004}
If a missing curve of $f$ is nonsingular then it is a coordinate line.

\item  \label {0000044}
If a contracting curve of $f$ has one place at infinity, then it is a coordinate line.

\end{enumerata}
\end{proposition}

\begin{proof}
Assertion \eqref{000004} follows from result 4.6 of \cite{Dai:bir}.
To prove \eqref{0000044}, consider a contracting curve $C$ of $f$ such that $C$ has one place at infinity.
We noted in \eqref{23748rhf63t} that $C$ is a nonsingular rational curve, so $C \isom \aff^1$ is clear.

Let us embed $\dom(f) = \aff^2$ in $X \isom \proj^2$, 
let $L$ be the function field of $X$ and $V$ the prime divisor of $L/\bk$ whose center on $X$ is the
closure of $C$ in $X$ (i.e., $V$ is the DVR $\Oeul_{X,\xi}$ where $\xi \in X$ is the generic point of $C$).
Also embed $\codom(f) = \aff^2$ in $Y \isom \proj^2$, and note that the center of $V$ on $Y$ is zero-dimensional,
since $C \in \Cont(f)$.

Consider the Kodaira dimension $\kappa(V)$ as defined in the introduction of Section 2 of \cite{Gan:Kod}.
Then $\kappa(V)<0$ by \cite[2.1]{Gan:Kod} and the fact that the center of $V$ on $Y$ is zero-dimensional; 
so $C$ is a coordinate line by \cite[2.4]{Gan:Kod}.
\end{proof}

\begin{corollary}  \label {ie123918209ejd10}
Let $C, C'$ be curves in $\aff^2$ such that $C \isom \aff^1 \isom C'$, and 
suppose that there exists a birational morphism $f : \aff^2 \to \aff^2$ such that $f(C)$ is a dense subset of $C'$.
Then $f(C)=C'$. Moreover, if one of $C,C'$ is a coordinate line then both are coordinate lines.
\end{corollary}

\begin{proof}
It is a simple fact that every dominant morphism $\aff^1 \to \aff^1$ is finite, hence surjective; so $f(C)=C'$.

If $C$ is a coordinate line then there exists a birational morphism $g : \aff^2 \to \aff^2$ such that $C \in \Miss(g)$
(choose a coordinate system $(X,Y)$ such that $C=\bZ(X)$, and take $g(x,y) = (x,xy)$);
then $C' \in \Miss(f \circ g)$ by \ref{Doifqrqkkdsjf8},
so \ref{938hdb67-gdhifa}\eqref{000004} implies that $C'$ is a coordinate line.

If $C'$ is a coordinate line then there exists a birational morphism $g : \aff^2 \to \aff^2$ such that $C' \in \Cont(g)$
(choose $(X,Y)$ such that $C' = \bZ(X)$ and take $g(x,y)=(x,xy)$); then $C \in \Cont(g \circ f)$,
so \ref{938hdb67-gdhifa}\eqref{0000044} implies that $C$ is a coordinate line.
\end{proof}

\begin{lemma}  \label {dfiuq234913[2q0k}
Let $\aff^2 \xrightarrow{ f } \aff^2 \xrightarrow{ g } \aff^2$ be 
birational morphisms.
\begin{enumerata}

\item If $\Miss(f) \subseteq \Cont(g)$ then $\Miss(f)$ is admissible.

\item If $\Cont(g) \subseteq \Miss(f)$ then $\Cont(g)$ is admissible.

\end{enumerata}
\end{lemma}

\begin{proof}
Applying statement~\eqref{23748rhf63t} in \ref{diofj;askdjf} to the morphism $g$ gives:
\begin{equation}  \label{dkfjpq83u9123[0i'dfop}
\begin{minipage}{.9\textwidth} \it
\item $D' = \displaystyle \sum_{C \in \Cont(g)} \!\!\! C$\ \ is
an SNC-divisor of $\aff^2$ and the graph $\Geul(D',\aff^2)$ is a forest.
\end{minipage}
\end{equation}
If $\Cont(g) \subseteq \Miss(f)$ then each element of $\Cont(g)$
is a {\it nonsingular\/} missing curve of $f$, and hence 
is a coordinate line by \ref{938hdb67-gdhifa}\eqref{000004};
then \eqref{dkfjpq83u9123[0i'dfop} implies that
$\Cont(g)$ is admissible, so (b) is proved.

If $\Miss(f) \subseteq \Cont(g)$ then, by \eqref{dkfjpq83u9123[0i'dfop},
$D = \sum_{C \in \Miss(f)} C$ is an SNC-divisor of $\aff^2$
and the graph $\Geul(D,\aff^2)$ is a forest; in particular each missing
curve of $f$ is nonsingular and hence
is a coordinate line by \ref{938hdb67-gdhifa}\eqref{000004};
so $\Miss(f)$ is admissible and (a) is proved.
\end{proof}

\begin{lemma}   \label {djf9q823j;d8zxcmv}
Let $\aff^2 \xrightarrow{ f } \aff^2 \xrightarrow{ g } \aff^2$
be birational morphisms.
\begin{enumerata}

\item If $\Miss(f) \nsubseteq \Cont(g)$ then $q( g \circ f ) > q(g)$.

\item If $\Cont(g) \nsubseteq \Miss(f)$ then $c( g \circ f ) > c(f)$.

\end{enumerata}
\end{lemma}

\begin{proof}
(a) It is clear that $\Miss(g) \subseteq \Miss( g \circ f )$.
If $C$ is a missing curve of $f$ which is not contracted by $g$ 
then the closure $\overline{ g(C) }$ of $g(C)$ is a missing curve of
$g \circ f$ but not a missing curve of $g$, so 
$\Miss(g) \subset \Miss( g \circ f )$ and hence $q(g) < q( g \circ f )$.

(b) We have $\Cont(f) \subseteq \Cont( g \circ f )$.
If $C$ is a contracting curve of $g$ which is not a missing curve of $f$
then there exists a curve $C'\subset\aff^2$ such that $f(C')$ is a dense
subset of $C$.  Then $C'$ is a contracting curve of $g \circ f$ but not 
one of $f$, so $\Cont(f) \subset \Cont( g \circ f )$ and hence
$c(f) < c( g \circ f )$.
\end{proof}

In \ref{sdkfqp923wr2390094rjdfk} and \ref{new-sdkfqp923wr2390094rjdfk},
we write $\#\Gamma$ for the number of irreducible components of a closed set $\Gamma$,
and $\Gamma_\tau = \bigcup_{C \in \Miss(\tau)}C$ for any birational morphism $\tau$ of nonsingular surfaces.

\begin{lemma} \label {sdkfqp923wr2390094rjdfk}
Let $f : \aff^2 \to \aff^2$ be a birational morphism
and $\Gamma$ a union of missing curves of $f$ such that
\begin{equation}  \label{difuqp29334343iucivju/}
\text{each missing curve of $f$ is either included in $\Gamma$
or disjoint from $\Gamma$.}
\end{equation}
Then $\# f^{-1}( \Gamma ) = \#\Gamma$ and 
$f$ can be factored as $\aff^2 \xrightarrow{g} \aff^2 \xrightarrow{h} \aff^2$,
where $g,h$ are birational morphisms, $\Gamma_h=\Gamma$ and $\Gamma_g \cap \exc(h) = \emptyset$.
\end{lemma}

Result \ref{sdkfqp923wr2390094rjdfk} is an immediate consequence of:

\begin{proposition}  \label {new-sdkfqp923wr2390094rjdfk}
Let $f : X \to Y$ be a birational morphism where
$X,Y$ are nonsingular affine surfaces and 
let $\Gamma \subset Y$ be a union of missing curves of $f$ satisfying~\eqref{difuqp29334343iucivju/}.
Then the following hold.
\begin{enumerata}

\item $f$ can be factored as $X \xrightarrow{g} W \xrightarrow{h} Y$ where
$W$ is a nonsingular affine surface, $g,h$ are birational morphisms,
$\Gamma_h=\Gamma$, $c(h) = \# f^{-1}(\Gamma)$ and $\Gamma_g \cap \exc(h) = \emptyset$.

\item If $X,Y$ are factorial with trivial units then $\# f^{-1}( \Gamma ) = \#\Gamma$ and,
in part~{\rm (a)}, $W$ can be chosen to be factorial with trivial units.

\item If $X=\aff^2$ and $Y$ is factorial, 
then $Y=\aff^2$ and we can choose $W=\aff^2$ in part~{\rm (a)}.

\end{enumerata}
\end{proposition}

\begin{proof} 
(a) We may choose a commutative diagram \eqref{mindec}
satisfying $n=n(f)$ and in which the blowings-up $\pi_1, \dots, \pi_n$ are
ordered in such a way that the points over $\Gamma$ are blown-up first,
i.e., there exists $m \in \{0, \dots, n\}$ such that 
$$
\setspec{ i \in \{1, \dots, n\} }{ P_i \in \Gamma \text{ or $P_i$ is infinitely near a point of $\Gamma$} }
= \{ 1, \dots, m\} .
$$
Refer to \ref{diofj;askdjf} for the notations.
If $0 \le j \le k \le n$ and $D \subset Y_j$ is a curve or a union of curves,
we write $\widetilde D^{Y_k}$ for the strict transform of $D$ on $Y_k$.
Let $J$ be the set of $j \in \{ 1, \dots, m \}$ such that $\widetilde E_j^{Y_n} \cap X = \emptyset$
(recall that $X$ is an open subset of $Y_n$) and define
\begin{equation} \label {2k3jb234i235389fsu}
W = Y_m \setminus 
\big( \widetilde\Gamma^{Y_m}\ \cup\  \bigcup_{ j \in J } \widetilde E_j^{Y_m} \big) .
\end{equation}
Then $W$ is a nonsingular surface and $f$ factors as 
$X \xrightarrow{g} W \xrightarrow{h} Y$ where
$g,h$ are birational morphisms, $\Gamma_h=\Gamma$
and $\Cont(h) = \setspec{ \widetilde E_i^{Y_m}\cap W }{ i \in \{ 1, \dots, m \} \setminus J}$; thus
\begin{equation}  \label {283471h8d6fy756545hki56ho4}
q(h) = \#\Gamma \quad\text{and}\quad c(h) = \# f^{-1}(\Gamma) .
\end{equation}
Let $\Gamma' = C_1 \cup \dots \cup C_s$
where $C_1, \dots, C_s \subset Y$ are the missing curves of $f$ not included in $\Gamma$;
then $\Gamma_f = \Gamma \cup \Gamma'$ and  $\Gamma \cap \Gamma' = \emptyset$.
Moreover,
\begin{equation} \label {8d88ds89sfgd98sg7s79sfu}
\Miss(g) = \setspec{ \widetilde C_i^{Y_m} \cap W }{ i = 1, \dots, s}.
\end{equation}
Indeed, consider $C \in \Miss(g)$.
If $h(C)$ is a point then $C= \widetilde E_j^{Y_m}\cap W$ for 
some  $j \in \{ 1, \dots, m \}$, and in fact  $\widetilde E_j^{Y_n}\cap X = \emptyset$ (so $j \in J$) 
otherwise $C$ would not be a missing curve of $g$;
then \eqref{2k3jb234i235389fsu} implies that $\widetilde E_j^{Y_m}\cap W = \emptyset$, which is absurd. 
So $h(C)$ is a dense subset of a curve $C_* \subset Y$. Now $C_* \in \Miss(f)$ by \ref{Doifqrqkkdsjf8},
and \eqref{2k3jb234i235389fsu} implies that $C \nsubseteq \widetilde\Gamma^{Y_m}$,
hence $C_* \nsubseteq \Gamma$;
so $C_* \subseteq \Gamma'$ and consequently $C = \widetilde C_i^{Y_m} \cap W$ for some $i \in \{1, \dots, s\}$.
This proves ``$\subseteq$'' in \eqref{8d88ds89sfgd98sg7s79sfu},
and ``$\supseteq$'' easily follows from \ref{Doifqrqkkdsjf8}.

From \eqref{8d88ds89sfgd98sg7s79sfu}, we deduce that
$\Gamma_g \cap \exc(h) \subseteq h^{-1}(\Gamma') \cap h^{-1}(\Gamma)$, so
\begin{equation} \label {jf923cp19dkQ02}
\Gamma_g \cap \exc(h) = \emptyset.
\end{equation}
To prove (a), there only remains to show that $W$ is affine.
Since $X$ is affine, \ref{788hkckdssljkhjh25} implies that $W\setminus \Gamma_g$ is affine;
as  (by \eqref{jf923cp19dkQ02}) $\exc(h) \subset W\setminus \Gamma_g$,
\begin{equation}  \label {jkxv1yy7e328763rf-.ee}
\text{no contracting curve of $h$ is a complete curve.}
\end{equation}

Embed $Y_0$ in a nonsingular projective surface $\overline Y_0$ and enlarge diagram~\eqref{mindec}
as follows:
\begin{equation*}  
\xymatrix@R=10pt{
{\overline Y_n} \ar[r]^{\bar\pi_n} & \cdots \ar[r]^{\bar\pi_{m+1}} & {\overline Y_m} \ar[r]^{\bar\pi_m} &
\cdots  \ar[r]^{\bar\pi_1} & {\overline Y_0} \\
Y_n \ar @{^{(}->}[u] \ar[r]^{\pi_n} & \cdots \ar[r]^{\pi_{m+1}} & Y_m \ar @{^{(}->}[u] \ar[r]^{\pi_m} &
\cdots  \ar[r]^{\pi_1} & Y_0 \ar @{^{(}->}[u] \\
X \ar @{^{(}->}[u] \ar[rr]^{ g }  && W \ar @{^{(}->}[u] \ar[rr]^{ h } &&  Y \ar @{=}[u] 
}
\end{equation*}
Let $D_i = \overline Y_i \setminus Y_i$ ($0 \le i \le n$).
Since $Y=Y_0$ is affine, $D_0$ is a nonempty connected union of curves and each irreducible component of 
$\overline \Gamma$ meets $D_0$ (where $\overline \Gamma$ denotes the closure of $\Gamma$ in $\overline Y_0$).
It follows that $D_m$ is a nonempty connected union of curves and that each irreducible component of 
$\overline{ \widetilde \Gamma^{Y_m} }$ meets $D_m$.
Let us argue that
\begin{equation}  \label {ll23o3o32o9r238yf}
\text{$W$ is connected at infinity.}
\end{equation}
Suppose that \eqref{ll23o3o32o9r238yf} is false;
then $\overline Y_m \setminus W$ is not connected, so some connected
component $\Ceul$ of $\overline Y_m \setminus W$ is disjoint from $D_m$.
Then $\Ceul$ does not contain any irreducible component of $\widetilde \Gamma^{Y_m}$.
By \eqref{2k3jb234i235389fsu}, it follows that $\Ceul \subseteq \bigcup_{ j \in J } \widetilde E_j^{Y_m}$.

We have
$\overline Y_m \setminus \big( W \setminus \Gamma_g \big) =
\overline{ \widetilde C_1^{Y_m} } \cup \dots \cup \overline{ \widetilde C_s^{Y_m} } \cup ( \overline Y_m \setminus W )$
by \eqref{8d88ds89sfgd98sg7s79sfu};
since $W \setminus \Gamma_g$ is affine,
$$
\overline{ \widetilde C_1^{Y_m} } \cup \dots \cup \overline{ \widetilde C_s^{Y_m} } \cup ( \overline Y_m \setminus W )
\ \ \text{is connected}.
$$
As  $\overline Y_m \setminus W$ is not connected and $\Ceul$ is a connected component of it,
some $\overline{ \widetilde C_i^{Y_m} }$ must meet $\Ceul$.
So there exist $i \in \{1, \dots, s\}$ and $j \in J$ such that 
$\overline{ \widetilde C_i^{Y_m} } \cap  \widetilde E_j^{Y_m} \neq \emptyset$. 
As $C_i \subseteq \Gamma'$, this implies that $P_j \in \Gamma'$ or $P_j$ is i.n.\ a point of $\Gamma'$;
since $j \le m$, we also have $P_j \in \Gamma$ or $P_j$ is i.n.\ a point of $\Gamma$;
so $\Gamma \cap \Gamma' \neq \emptyset$, a contradiction.
So \eqref{ll23o3o32o9r238yf} is true.

In view of \eqref{jkxv1yy7e328763rf-.ee}, \eqref{ll23o3o32o9r238yf} and the fact that $Y$ is affine,
applying \cite[2.2]{Dai:bir} to $h : W \to Y$ shows that $W$ is affine and concludes the proof of (a).

\medskip
\noindent (b) Assume that $X,Y$ are factorial and have trivial units;
then \cite[3.4]{Dai:bir} gives $q(h) \le c(h)$, so 
$\#\Gamma \le \# f^{-1}(\Gamma)$ by \eqref{283471h8d6fy756545hki56ho4}.
Since $\Gamma'$ also satisfies \eqref{difuqp29334343iucivju/}, it follows that
$\#\Gamma' \le \# f^{-1}(\Gamma')$.

By \ref{788hkckdssljkhjh25} we have $\cent(f) \subseteq \Gamma_f = \Gamma \cup \Gamma'$,
so $f^{-1}(\Gamma) \cup f^{-1}(\Gamma')$ is exactly the 
union of all contracting curves of $f$; as 
$f^{-1}(\Gamma) \cap f^{-1}(\Gamma') = \emptyset$,
we get $\#f^{-1}(\Gamma) + \#f^{-1}(\Gamma') = c(f)$.
We have $c(f)=q(f)$ by \cite[2.9]{Dai:bir}, and it is clear that
$q(f) = \#\Gamma + \#\Gamma'$, so
$$
\#\Gamma \le \# f^{-1}(\Gamma),\ \ 
\#\Gamma' \le \# f^{-1}(\Gamma')\ \ 
\text{ and }\ \ 
\#\Gamma + \#\Gamma' = \# f^{-1}(\Gamma)+ \# f^{-1}(\Gamma');
$$
consequently,
\begin{equation}  \label{dfj19234jqklefnd89czn}
\#\Gamma = \# f^{-1}(\Gamma)
\end{equation}
\begin{equation}  \label {8238h723grf6vxmfj}
q(h)=c(h) 
\end{equation}
where \eqref{8238h723grf6vxmfj} follows from \eqref{dfj19234jqklefnd89czn} and \eqref{283471h8d6fy756545hki56ho4}.
By \eqref{8238h723grf6vxmfj}, \eqref{ll23o3o32o9r238yf} and \cite[3.4]{Dai:bir},
we get that $W$ is factorial and has trivial units, which proves (b).

If $X=\aff^2$ and $Y$ is factorial then, by (b), $W$ may be chosen to be factorial;
then \cite[4.2]{Dai:bir} implies that $W$ and $Y$ are isomorphic to $\aff^2$,
which proves (c) and completes the proof of the Proposition.
\end{proof}

\begin{definition} \label {9f2039esdow9e} 
Let $f : X \to Y$ be a birational morphism of nonsingular surfaces.
Consider a diagram \eqref{mindec} satisfying $n=n(f)$ and with notation as in \ref{diofj;askdjf}
(for each $i \in \{1, \dots, n\}$,
$\pi_i : Y_i \to Y_{i-1}$ be the blowing-up of $Y_{i-1}$ at the point $P_i \in Y_{i-1}$).

\begin{enumerata}

\item Let $C$ be a missing curve of $f$.
For each $i = 0, 1, \dots, n$, let $C^{Y_i} \subset Y_i$ denote the strict transform
of $C$ on $Y_i$ ($C^{Y_0} = C$). Then we define the natural number
$$
n(f,C) = \text{ cardinality of the set } \setspec{ i }{ 1 \le i \le n,\ P_i \in C^{Y_{i-1}} } 
$$
and note that $n(f,C)$ depends only on $(f,C)$, i.e., is independent of the choice of diagram~\eqref{mindec}.
To indicate that $n(f,C) = k$, we say that ``$C$ is blown-up $k$ times''.

\item For each $i = 1, \dots, n$, let $\bar P_i \in Y_0$ be the image of $P_i$ by
$\pi_1 \circ \cdots \circ \pi_{i-1} : Y_{i-1} \to Y_0$.
For each $P \in Y$, define the natural number
$$
n(f,P) = \text{ cardinality of the set } \setspec{ i }{ 1 \le i \le n,\  \bar P_i = P } 
$$
and note that $n(f,P)$ depends only on $(f,P)$,
i.e., is independent of the choice of the diagram \eqref{mindec} used for defining it.

\end{enumerata}
\end{definition}

\begin{remarks}
Let $f : X \to Y$ be a birational morphism of nonsingular surfaces.

\begin{enumerata}

\item Let $C \in \Miss(f)$.
Then $n(f,C)=0 \Leftrightarrow C \cap f(X) = \emptyset$,
and $n(f,C)=1$ implies that there exists exactly one fundamental point of $f$ lying on $C$.
Note that if $X=\aff^2=Y$ then each missing curve contains at least one fundamental point
(\ref{dijfaksdjfall00000}\eqref{POdsifias2874wke}),
so each missing curve is blown-up at least once.

\item Let $P \in Y$. Then
$n(f,P)>0 \Leftrightarrow P \in \cent(f)$, 
where ``$\Leftarrow$'' is obvious and ``$\Rightarrow$'' follows from \eqref{1329jf9m39c37hf}.
It is also clear that $n(f) = \sum_{P \in Y} n(f,P)$.

\end{enumerata}
\end{remarks}

\begin{lemma} \label {90f20rdw9019jds}
Let $X \xrightarrow{f} Y \xrightarrow{g} Z$ be birational morphisms of nonsingular surfaces, and assume
that $n(g\circ f) = n(g)+n(f)$ or $X=Y=Z=\aff^2$.
\begin{enumerata}

\item  \label {70209ru29fyt2suddh29}  
Let $D \in \Miss(g)$; then $D \in \Miss(g\circ f)$ and $n(g\circ f, D) = n(g,D)$.

\item  \label {lwoj837456djd93}  
Let $C \in \Miss(f) \setminus \Cont(g)$ and let $D$ be the closure of $g(C)$ in $Z$.
Then:
\begin{itemize}

\item  $D \in \Miss(g \circ f)$ and $n(f,C) \le n(g\circ f, D)$;

\item $ n(f,C) = n(g\circ f, D) \implies C \cap \exc(g) = \emptyset$;

\item if $g(C)=D$ or $C\isom \aff^1$ then
$$ 
n(f,C) = n(g\circ f, D) \iff C \cap \exc(g) = \emptyset .
$$

\end{itemize}
\item For each $P \in Z$,\ \ $\displaystyle n(g \circ f, P) = n(g,P) + \sum_{Q \in g^{-1}(P)} n(f,Q)$.
\end{enumerata}
\end{lemma}

\begin{proof} 
If $X=Y=Z=\aff^2$ then $n(g\circ f) = n(g)+n(f)$ by \ref{1dijfaksdjfall00100};
so $n(g\circ f) = n(g)+n(f)$ holds in all cases.
Let $m=n(f)$ and $n=n(g)$. 
Choose commutative diagrams (I) and (II) as follows:
\begin{equation*} 
\xymatrix@R=12pt{
Y_{m}  \ar @{} [drr] | {\mbox{\small (I)}} \ar[r]^{\pi_{m}} & \cdots \ar[r]^{\pi_1} & Y_0\\
X \ar @{^{(}->}[u] \ar[rr]_{ f }  &&  Y \ar @{=}[u] 
}
\quad
\xymatrix@R=12pt{
Z_{n} \ar @{} [drr] | {\mbox{\small (II)}}  \ar[r]^{\rho_{n}} & \cdots \ar[r]^{\rho_1} & Z_0\\
Y \ar @{^{(}->}[u] \ar[rr]_{ g }  &&  Z \ar @{=}[u] 
}
\end{equation*}
and use them to build the commutative diagram
\begin{equation} \tag{III}
\xymatrix@R=12pt{
Z_{n+m} \ar[r]^{\rho_{n+m}} & \cdots \ar[r]^{\rho_{n+2}} & Z_{n+1} \ar[r]^{\rho_{n+1}} & Z_n \ar[r]^{\rho_n} & \cdots \ar[r]^{\rho_1} & Z_0 \\
Y_{m}  \ar @{^{(}->}[u] \ar[r]^{\pi_{m}} & \cdots \ar[r]^{\pi_2} & Y_1 \ar @{^{(}->}[u] \ar[r]^{\pi_{1}} & Y_0\ar @{^{(}->}[u] \\
X \ar @{^{(}->}[u] \ar[rrr]^{ f }  &&&  Y \ar @{=}[u]  \ar[rr]^{ g }  &&  Z \ar @{=}[uu] 
}
\end{equation}
In the three diagrams, ``$\hookrightarrow$'' are open immersions, $Y_i, Z_i$ are nonsingular surfaces,
$Y_i \xrightarrow{ \pi_i } Y_{i-1}$ is the blowing-up of $Y_{i-1}$ at a point $P_i \in Y_{i-1}$
and
$Z_i \xrightarrow{ \rho_i } Z_{i-1}$ is the blowing-up of $Z_{i-1}$ at a point $Q_i \in Z_{i-1}$.
Moreover, $Y_{i-1} \hookrightarrow Z_{n+i-1}$ maps $P_i$ on $Q_{n+i}$ (let us simply write $P_i=Q_{n+i}$).
Diagrams (I) and (II) are minimal in the sense that $n(f)=m$ and $n(g)=n$;
since $n(g\circ f) = n(f)+n(g)=m+n$, it follows that (III) is also minimal.

\eqref{70209ru29fyt2suddh29}
Let $D \in \Miss(g)$; then $D \in \Miss(g \circ f)$ by \ref{Doifqrqkkdsjf8}\eqref{qo8723yrd187ryqaj}.
Let $D^{Z_i} \subset Z_i$ be the strict transform of $D \subset Z_0$ on $Z_i$.
As $D^{Z_n} \subseteq Z_n \setminus Y_0$ and 
$\cent(\rho_{n+1} \circ \cdots \circ \rho_{n+m}) = \cent(\pi_1 \circ \cdots \circ \pi_{m}) \subset Y_0$,
we see that
\begin{equation}  \label {8923ejxbv6c3eyrf}
\setspec{ i }{ 1 \le i \le n+m,\ Q_i \in D^{Z_{i-1}} } = \setspec{ i }{ 1 \le i \le n,\ Q_i \in D^{Z_{i-1}} }.
\end{equation}
Since $n(g\circ f,D)$ (resp.~$n(g,D)$) is by definition the cardinality of the set in the 
lhs (resp.~rhs) of \eqref{8923ejxbv6c3eyrf},  we have $n(g\circ f,D) = n(g,D)$.

\eqref{lwoj837456djd93} Let $C \in \Miss(f) \setminus \Cont(g)$ and let $D$ be the closure of $g(C)$ in $Z$.
Then $D \in \Miss(g \circ f)$ by \ref{Doifqrqkkdsjf8}\eqref{qo8723yrd187ryqaj}.
Define $D^{Z_i} \subset Z_i$ as before, then
\begin{equation}  \label {g251e98fjc87tg}
\setspec{ i }{ n+1 \le i \le n+m,\ Q_i \in D^{Z_{i-1}} }
\subseteq \setspec{ i }{ 1 \le i \le n+m,\ Q_i \in D^{Z_{i-1}} }.
\end{equation}
Since $n(f,C)$ (resp.~$n(g\circ f,D)$) is the cardinality of the set in the 
lhs (resp.~rhs) of \eqref{g251e98fjc87tg}, 
we have $n(f,C) \le n(g\circ f, D)$, and moreover
\begin{equation}  \label {sjdhdg376t7856s}
n(f,C) \neq n(g\circ f, D)
\end{equation}
is equivalent to
\begin{equation}  \label {2w2ggft12fh2h12jiop}
\setspec{ i }{ 1 \le i \le n,\ Q_i \in D^{Z_{i-1}} } \neq \emptyset.
\end{equation}
By minimality of diagram (II) together with \eqref{1329jf9m39c37hf},
\eqref{2w2ggft12fh2h12jiop} is equivalent to
\begin{equation}  \label {987x766dv43h}
D \cap \cent(g) \neq \emptyset .
\end{equation}
Now
\begin{equation}  \label {sg72167263ru8q9kkrfj}
C \cap \exc(g) \neq \emptyset
\end{equation}
implies \eqref{987x766dv43h} and, if $g(C)=D$, the converse is true.
So we have shown that
\begin{equation}  \label {88r23hr26265215r}
n(f,C) = n(g\circ f, D) \implies C \cap \exc(g) = \emptyset,
\end{equation}
and that the converse holds whenever $g(C)=D$.
Finally, we observe that if $C \isom \aff^1$ then the dominant morphism
$C \xrightarrow{g} D$ is necessarily finite, hence surjective,
so the converse of \eqref{88r23hr26265215r} is true whenever $C \isom \aff^1$.
This proves (b).

To prove (c), define $\bar Q_i = ( \rho_1 \circ \cdots \circ \rho_{i-1})(Q_i) \in Z_0$ for all $i = 1, \dots, m+n$
and observe that the trivial equality (for any $P \in Z$)
$$
| \setspec{ i }{ \bar Q_i = P } |
= | \setspec{ i }{ i \le n \text{ and } \bar Q_i = P } | + | \setspec{ i }{ i > n \text{ and }  \bar Q_i = P } |
$$
is the desired equality.
\end{proof}

\section{Compositions of simple affine contractions}
\label {SECCompositionsofsimpleaffinecontractions}

Let $\bk$ be an algebraically closed field of arbitrary characteristic, and let $\aff^2=\aff^2_\bk$.
As in the introduction, we write $\BirA$ for the monoid of birational endomorphisms $f : \aff^2 \to \aff^2$,
and we declare that $f,g \in \Bir(\aff^2)$ are equivalent ($f\sim g$) if $u \circ f \circ v = g$ for
some $u,v \in \Aut(\aff^2)$.
The equivalence class of $f \in \BirA$ is denoted $[f]$.
Note that the conditions $f \sim f'$ and $g \sim g'$ do NOT imply that $f \circ g \sim f' \circ g'$. 

The aim of this section is to describe the subsets
$S_{\text{\rm w}} \supset S_{\text{\rm a}} \supset S_{\text{\rm aa}}$ of $\BirA$ defined by:
\begin{align*}
S_{\text{\rm w}} & = \setspec{ f \in \BirA }{ \text{$\Miss(f)$ is weakly admissible} }, \\
S_{\text{\rm a}}  & = \setspec{ f \in \BirA }{ \text{$\Miss(f)$ is admissible} }, \\
S_{\text{\rm aa}} & = \setspec{ f \in \BirA }{ \text{both $\Miss(f)$ and $\Cont(f)$ are admissible} }
\end{align*}
(refer to \ref{weaklyioufq923lksmdjjj}, \ref{weak-dpfq32748rjkdkd},
\ref{ioufq923lksmdjjj} and \ref{dpfq32748rjkdkd});
note that these sets are not closed under composition of morphisms.
We learn at a relatively early stage 
(see \ref{dkfjaiuer834134k1j3}\eqref{78238723fi111101})
that each element of $S_{\text{\rm w}}$ 
is a composition of simple affine contractions
(SACs are defined in the introduction and again in \ref{dfuqp923jwe;li}).
However, an arbitrary composition of SACs does not necessarily belong to 
$S_{\text{\rm w}}$ (resp.\  $S_{\text{\rm a}}$, $S_{\text{\rm aa}}$),
so in each of the three cases one has to determine which compositions of SACs give
the desired type of endomorphism.
The answer is given in Theorem~\ref{du-239412948ynf}, which is the main result of this section.

\medskip
The material of \ref{new-dkfqi398}--\ref{nouveau-dkfjaiuer834134k1j3}\eqref{8734yt9873y4reyw}
can be found in \cite{Dai:bir} and \cite{Dai:trees};
everything else appears to be new.

\medskip
As before, we have $\aff^2 = \Spec A$ where $A = \kk2$ is fixed throughout,
and by a coordinate system of $\aff^2$ we mean a pair $(X,Y) \in A\times A$ satisfying $A=\bk[X,Y]$
(see the introduction of Section~\ref{Sec:AdmissConfig}).

\begin{parag}  \label {new-dkfqi398} 
Let $\Cgoth$ temporarily denote the set of coordinate systems of $\aff^2$.
Given an element $\cgoth =(X,Y)$ of $\Cgoth$,
consider the $\bk$-homomorphism $A \to A$ given by $X \mapsto X$ and $Y \mapsto XY$;
this homomorphism determines a morphism $\aff^2 \to \aff^2$ which we denote $\alpha_\cgoth$;
clearly, $\alpha_\cgoth \in \Bir( \aff^2)$.
Note that if $\cgoth_1, \cgoth_2 \in \Cgoth$ then
$\alpha_{\cgoth_1}^m \sim \alpha_{\cgoth_2}^m$ for every $m \ge 1$.
So, for each $m \ge 1$, the equivalence class $[\alpha_\cgoth^m]$ of $\alpha_\cgoth^m$ is independent
of the choice of $\cgoth \in \Cgoth$.
\end{parag}

\begin{definition}    \label {dfuqp923jwe;li}
A birational morphism $\aff^2 \to \aff^2$
is called a {\it simple affine contraction (SAC)\/} if it is equivalent to $\alpha_\cgoth$
for some (hence for every) coordinate system $\cgoth$ of $\aff^2$.
\end{definition}

Note that if $f$ is a SAC and $\cgoth \in \Cgoth$  then $f \sim \alpha_\cgoth$, but $f^2$ need not be equivalent
to $\alpha_\cgoth^2$.

\medskip
For readers who like to identify $\aff^2$ with $\bk^2$, we note that $\alpha_\cgoth$ corresponds to 
the map $\bk^2 \to \bk^2$, $(x,y) \mapsto (x,xy)$,
and that the SACs are obtained by composing this map both sides with automorphisms.
See \ref{43fg6gxcfzx}.

\begin{lemma}  \label {dkjfaidsukkkdr333}
\mbox{\ }
\begin{enumerata}

\item  \label{BE4.10_LT4.1}
A birational morphism $f : \aff^2 \to \aff^2$ is a SAC
if and only if $n(f)=1$.

\item \label{djfoiqwuefkd6767676}
If $f : \aff^2 \to \aff^2$ is a SAC then $f$ has one missing  curve $L$
and one fundamental point $P$; moreover, $L$ is a coordinate line
and $P\in L$.

\item \label {88e8e8e8w8e8dv8g8e88} 
Let $L \subset \aff^2$ be a coordinate line and $P\in L$ a point.
Let $X \xrightarrow{\pi} \aff^2$ be the blowing-up of $\aff^2$ at $P$
and $U \subset X$ the complement of the strict transform of $L$.
Then $U \isom \aff^2$ and the composition
$\aff^2 \xrightarrow{ \isom } U \hookrightarrow X \xrightarrow{\pi} \aff^2$
is a SAC with missing curve $L$ and fundamental point $P$.

\item \label{djflahds3y476476354}
If $f,g : \aff^2 \to \aff^2$ are two SAC with the same missing curve
and  the same fundamental point then there exists an automorphism
$\theta : \aff^2 \to \aff^2$ such that $g = f \circ \theta$:
$$
\xymatrix{
{\aff^2} \ar[dr]_{g} \ar[rr]_{\theta}^{\isom} && {\aff^2} \ar[dl]^{f}  \\
& {\aff^2}
}
$$

\item  \label {BE4.11}
Let $\cgoth$ be a coordinate system of $\aff^2$ and $\alpha_\cgoth \in \BirA$ as in \ref{new-dkfqi398}.
For $f \in \BirA$, 
$$
q(f)=1 \iff c(f)=1 \iff f \sim \alpha_\cgoth^{n(f)} .
$$

\end{enumerata}
\end{lemma}

\begin{proof}
Assertion~\eqref{BE4.10_LT4.1} is \cite[4.10]{Dai:bir} or \cite[4.1]{Dai:trees};
assertions \eqref{djfoiqwuefkd6767676}--\eqref{djflahds3y476476354}
are trivial; \eqref{BE4.11} is \cite[4.11]{Dai:bir} together with the
fact \ref{dijfaksdjfall00000}\eqref{000001} that  $q(f)=c(f)$.
\end{proof}

\begin{smallremark}
Assertion \ref{dkjfaidsukkkdr333}\eqref{BE4.11} can be phrased as follows:
given an integer $m \ge 1$ and a coordinate system $\cgoth$ of $\aff^2$, 
the equivalence class  $[\alpha_\cgoth^m]$ of $\alpha_\cgoth^m$ is:
\begin{align*}
[\alpha_\cgoth^m] &= \setspec{ f \in \BirA }{ q(f)=1 \text{ and } n(f)=m } \\
 &= \setspec{ f \in \BirA }{ c(f)=1 \text{ and } n(f)=m }.
\end{align*}
\end{smallremark}

\begin{corollary}  \label {uniquemissingcoord}
If $f \in \BirA$ has a unique missing curve $C$, then $C$ is a coordinate line.
\end{corollary}

\begin{proof}
This follows from \ref{dkjfaidsukkkdr333}\eqref{BE4.11}.
One can also deduce it from \ref{938hdb67-gdhifa}\eqref{000004} and \ref{g64f4j3k21ld}.
\end{proof}

See \ref{9f2039esdow9e} for the definition of the phrase ``$L$ is blown-up only once''.

\begin{lemma}  \label {nouveau-dkfjaiuer834134k1j3}
Let $f \in \BirA$.
\begin{enumerata}

\item \label {8734yt9873y4reyw}
Suppose that some missing curve $L$ of $f$ is blown-up only once.
Then $L$ is a coordinate line.  Moreover, if $P \in L$ is the unique fundamental point of $f$ which is on $L$ and
$\gamma \in \BirA$ is a SAC with missing curve $L$ and fundamental point $P$, then $f$ factors as 
$\aff^2 \xrightarrow{f'} \aff^2 \xrightarrow{\gamma} \aff^2$ with $f' \in \BirA$.

\item \label {hd819t5fa3s5erh}
Suppose that $f$ factors as 
$\aff^2 \xrightarrow{f'} \aff^2 \xrightarrow{\gamma} \aff^2$ with $f',\gamma \in \BirA$,
where $\gamma$ is a SAC. Let $L$ be the missing curve of $\gamma$.
Then $L$ is a missing curve of $f$ which is blown-up only once.

\end{enumerata}
\end{lemma}  

\begin{proof}
\eqref{8734yt9873y4reyw} is an improvement of \cite[4.6]{Dai:trees}.
The proof of \cite[4.6]{Dai:trees} shows that $P$ is a nonsingular point of $L$;
then \cite[4.6]{Dai:bir} implies that $L$ is a coordinate line.
Choose a diagram~\eqref{mindec} for $f$ such that $n=n(f)$ and $P_1 = P$.
Let $L^{Y_1} \subset Y_1$ denote the strict transform of $L$ on $Y_1$
and define $W = Y_1 \setminus L^{Y_1} \subset Y_1$.
As $L$ is a missing curve of $f$ and is blown-up only once, the image of
$\aff^2 \hookrightarrow Y_n \xrightarrow{\pi_2 \circ \cdots \circ \pi_n} Y_1$ is included in $W$;
so $f$ factors as $\aff^2 \xrightarrow{g'} W \xrightarrow{h'} \aff^2$ where $g',h'$ are birational morphisms
and $h'$ is the composition $W \hookrightarrow Y_1 \xrightarrow{\pi_1} Y_0 = \aff^2$.
By \ref{dkjfaidsukkkdr333}\eqref{88e8e8e8w8e8dv8g8e88},
$W \isom \aff^2$ and the composition $\aff^2 \isom W \xrightarrow{h'} \aff^2$
is a SAC with missing curve $L$ and fundamental point $P$;
so $f$ factors as $\aff^2 \xrightarrow{g} \aff^2 \xrightarrow{h} \aff^2$ where $g,h \in \BirA$
and $h$ is a SAC with missing curve $L$ and fundamental point $P$.
By \ref{dkjfaidsukkkdr333}\eqref{djflahds3y476476354},
$h = \gamma \circ \theta$ for some $\theta \in \AutA$.
Then $f =  \gamma \circ \theta \circ g$ and we are done.

\eqref{hd819t5fa3s5erh}
By \ref{90f20rdw9019jds}\eqref{70209ru29fyt2suddh29}, $L \in \Miss(f)$ and 
$n(f,L) = n(\gamma \circ f', L) = n(\gamma, L) = 1$.
\end{proof}

\begin{proposition}  \label {dkfjaiuer834134k1j3}
Let $f : \aff^2 \to \aff^2$ be a birational morphism such that:
\begin{itemize}

\item[(i)] $f$ is not an isomorphism;

\item[(ii)] there exists a coordinate system of $\aff^2$ with respect to which
all missing curves of $f$ have degree $1$.

\end{itemize}
Then there exists a missing curve of $f$ which is blown-up only once.
Moreover, if $L$ is such a curve
and  $P \in L$ is the unique fundamental point of $f$ which is on $L$,
then the following hold.
\begin{enumerata}
\setlength{\itemsep}{1mm}

\item There exists a coordinate system $(X,Y)$ of $\aff^2$
such that $L=\bZ(X)$ and $P = (0,0)$,
and such that the union of the missing curves of $f$ is equal to the zero-set of
one of the following polynomials in $\bk[X,Y]$:
\begin{enumerata}
\setlength{\itemsep}{1mm}

\item $XY^m \prod_{i=1}^n (X-a_i)$, for some $m \in \{0,1\}$, $n\ge0$
and $a_1, \dots, a_n \in \bk$;

\item $X(X-1)^m \prod_{i=1}^n (Y-b_iX)$, for some $m \in \{0,1\}$, $n\ge0$
and $b_1, \dots, b_n \in \bk$.

\end{enumerata}

\item If $\gamma : \aff^2 \to \aff^2$ is a SAC with missing curve $L$ and fundamental
point $P$, then $f$ factors as $\aff^2 \xrightarrow{f'} \aff^2 \xrightarrow{\gamma} \aff^2$
where $f'$ is a birational morphism such that $\Miss(f')$ is admissible.

\item \label {78238723fi111101} $f$ is a composition of SACs.

\end{enumerata}
\end{proposition}

\begin{proof} 
By \cite[4.7]{Dai:trees}, $f = h \circ g$ where $g,h \in \BirA$ and $h$ is a SAC;
then \ref{nouveau-dkfjaiuer834134k1j3}\eqref{hd819t5fa3s5erh}
implies that some missing curve of $f$ is blown-up only once.

Let $L$ be a missing curve of $f$ which is blown-up only once,
and let $P \in L$ be the unique fundamental point of $f$ which is on $L$.
Choose a coordinate system $(X,Y)$ of $\aff^2$
such that $L=\bZ(X)$ and $P = (0,0)$,
and with respect to which all missing curves of $f$ have degree $1$.
Define $\gamma_0 : \aff^2 \to \aff^2$ by  $\gamma_0(x,y) = (x,xy)$.
As $\gamma_0$ is a SAC with missing curve $L$ and fundamental point $P$,
\ref{nouveau-dkfjaiuer834134k1j3}\eqref{8734yt9873y4reyw}
implies that $f$ factors as
$\aff^2 \xrightarrow{f'} \aff^2 \xrightarrow{\gamma_0} \aff^2$, for some $f' \in \BirA$.
Let $\Gamma$ (resp.\ $\Gamma'$) be the union of the missing curves of $f$ (resp.\ of $f'$).
Then
\begin{equation}  \label{difu-q923ooefk0s3373773}
\Gamma = \bZ(X) \cup \overline{ \gamma_0( \Gamma' ) } 
\end{equation}
by \ref{Doifqrqkkdsjf8}.  In particular,
if $C$ is a missing curve of $f'$ then $\gamma_0(C)$ is included into
a line of degree $1$; from $\gamma_0(x,y)=(x,xy)$, it easily follows that
$C$ is either a vertical line $\bZ(X-a)$ (for some $a \in \bk$),
a horizontal line $\bZ(Y-b)$ (for some $b \in \bk$),
or a hyperbola $\bZ(X(\alpha  + \beta Y)-1)$ (for some $\alpha,\beta \in \bk$, $\beta\neq0$),
where in fact the last case cannot occur because $C$ has one place at infinity by
\ref{dijfaksdjfall00000}\eqref{000003}.  So
\begin{equation}  \label{66222661626263t56}
\textit{each missing curve of $f'$ is either a vertical line or a horizontal line.}
\end{equation}
In particular, all missing curves of $f'$ have degree~$1$.  It follows from the first part
of the proof that
\begin{equation}   \label{dfrq293423789374-0349fkjodf}
\textit{if $f'$ is not an isomorphism, some missing curve of $f'$
is blown-up only once.}
\end{equation}
Let $h$ (resp.\ $v$) be the number of missing curves of $f'$ which are horizontal
(resp.\ vertical) lines.  Then $\min(h,v) \le 1$, otherwise
(by \ref{dijfaksdjfall00000}\eqref{000006})
every missing curve of $f'$ would contain at least two fundamental points of $f'$,
and hence would be blown-up at least twice, contradicting \eqref{dfrq293423789374-0349fkjodf}.
Statement \eqref{66222661626263t56} together with $\min(h,v) \le 1$ imply
that $\Miss(f')$ is admissible, which proves the special case ``$\gamma=\gamma_0$'' 
of assertion~(b); in view of \ref{dkjfaidsukkkdr333}\eqref{djflahds3y476476354},
it follows that (b) is true.

If $h\le 1$ then $\Gamma'$ is the zero-set of
$(Y-b)^h \prod_{i=1}^v (X-a_i)$, for some $b \in \bk$ and $a_1, \dots, a_v \in \bk$.
Then, by~\eqref{difu-q923ooefk0s3373773}, $\Gamma$ is the zero-set of
$X(Y-bX)^h \prod_{i=1}^v (X-a_i)$.
Replacing the coordinate system $(X,Y)$ by $(X,Y-bX)$, we see that (a-i) is satisfied.

If $v \le 1$ then $\Gamma'$ is the zero-set of
$(X-a)^v \prod_{i=1}^h (Y-b_i)$, for some $a \in \bk$ and $b_1, \dots, b_h \in \bk$.
Then $\Gamma$ is the zero-set of $X(X-a)^v \prod_{i=1}^h (Y-b_iX)$.
If $a=0$ or $v=0$ then $\Gamma$ is the zero-set of $X(X-1)^0 \prod_{i=1}^h (Y-b_iX)$, 
so (a-ii) holds;
if $a \neq 0$ and $v \neq 0$ then $\Gamma$ is the zero-set of 
$X(X-a) \prod_{i=1}^h (Y-b_iX)$, so (a-ii) holds after replacing $(X,Y)$ by $(a^{-1}X,Y)$.

So assertion~(a) is true.
To prove assertion~(c), consider the factorization $f = \gamma \circ f'$ given by (b).
Since  $n(\gamma)=1$ by \ref{dkjfaidsukkkdr333}\eqref{BE4.10_LT4.1},
we have $n(f') = n(f)-1$ by \ref{1dijfaksdjfall00100}.
Moreover, the fact that $\Miss(f')$ is admissible implies, by \ref{dpfq32748rjkdkd},
that there exists a coordinate system of $\aff^2$ with respect to which
all missing curves of $f'$ have degree $1$.
It is clear that (c) follows by induction on $n(f)$.
\end{proof}

\begin{remark}  \label {ClearUpConfusion}
We stress that assumption~(ii) of \ref{dkfjaiuer834134k1j3}
is strictly stronger than ``all missing curves are coordinate lines''.
Indeed, there exists an irreducible element $f \in \BirA$ with three missing curves,
these being the lines $\bZ(X+Y)$ and $\bZ(X-Y)$ and the parabola $\bZ(Y-X^2)$:
\begin{equation*}
\setlength{\unitlength}{1mm}
\aff^2 \xrightarrow{\ \ f\ \ } \aff^2 = 
\raisebox{-8\unitlength}{\fbox{\begin{picture}(30,20)(-15,-5)
\put(0,0){\circle*{1}}
\put(6.7,6.7){\circle*{1}}
\put(-6.6,6.6){\circle*{1}}
\put(-5,-5){\line(1,1){20}}
\put(-15,15){\line(1,-1){20}}
\qbezier(0,0)(-5,0)(-10,15)
\qbezier(0,0)(5,0)(10,15)
\end{picture}}}
\end{equation*}
(an example of such an $f$, due to Russell, appeared in \cite[4.7]{Dai:bir}).
Here, each missing curve is a coordinate line, and hence has degree $1$ with respect to a suitable choice of
coordinate system. However, these three lines are not {\it simultaneously\/} rectifiable, so $f$ does not satisfy
assumption~(ii) of \ref{dkfjaiuer834134k1j3} (it does not satisfy the conclusion either: 
since $f$ is not a SAC and is irreducible, it is not a composition of SACs).

Also note that, by \ref{weak-dpfq32748rjkdkd}, assumption~(ii) of \ref{dkfjaiuer834134k1j3} is equivalent to 
``$\Miss(f)$ is weakly admissible''.
\end{remark}

\begin{parag} \label {6f78gt6h5yghjd}
From now-on, and until the end of section~\ref{SECCompositionsofsimpleaffinecontractions},

\begin{center}
we fix a coordinate system $\cgoth=(X,Y)$ of $\aff^2$.
\end{center}

This allows us to identify $\aff^2$ with $\bk^2$.
See \ref{43fg6gxcfzx} for the notation $\bZ(F)$ and for our convention
regarding the definition of morphisms using coordinates.
\end{parag}

\begin{parag}  \label {27384ygfq7239j}
In \ref{dkfjq9384021awe6446}--\ref{2376123y234ygd76} below, we define
three submonoids of $\BirA$, denoted $\bH_\cgoth$, $\bG_\cgoth$ and $\bV_\cgoth$, respectively.
The subscript $\cgoth$ reminds us that these sets depend on the choice of $\cgoth$ made in \ref{6f78gt6h5yghjd}.
Since $\cgoth$ is fixed until the end of this section,
there is no harm in omitting it and writing simply $\bH$, $\bG$, and $\bV$.
It is clear from the definitions below that these three monoids are included in the submonoid of $\BirA$ 
generated by SACs and automorphisms.
\end{parag}

\begin{subparag}  \label {dkfjq9384021awe6446}
Given $m\in\Nat$ and $p \in \bk[Y]$ such that\footnote{We adopt the
convention that the zero polynomial has
degree $-\infty$; consequently, the condition $\deg p < 0$ is
equivalent to $p$ being the zero polynomial
(so $h_{0,p}=h_{0,0}$ is the identity map).}
$\deg p < m$, define $h_{m,p} \in \BirA$ 
by $h_{m,p}(x,y) = (xy^m+p(y), y)$.
Observe that $h_{m,p}$ is equivalent to $\gamma^m$, where $\gamma$ is the
SAC given by $(x,y) \mapsto (xy,y)$; consequently, 
$n(h_{m,p})= n( \gamma^m ) = m n(\gamma)$, i.e.,
$$
n(h_{m,p})= m .
$$
Define $\bH = \bH_\cgoth = \setspec{ h_{m,p} }{\text{$m \in \Nat$, $p \in \bk[Y]$ and $\deg p < m$} }$.
It is easily verified that $\bH$ is a submonoid of $\BirA$.
\end{subparag}

\begin{subparag}  \label{difuqp934913kmq34}
Let $\bM$ be the multiplicative monoid whose elements are
the $2\times 2$ matrices
$M=\left( \begin{smallmatrix} i & j \\ k & \ell \end{smallmatrix} \right)$
with $i,j,k,\ell \in \Nat$ and $i\ell - jk=\pm1$.
It is easily verified that $\bM$ is generated by 
$\left\{
\left( \begin{smallmatrix} 1 & 1 \\ 0 & 1 \end{smallmatrix} \right),
\left( \begin{smallmatrix} 0 & 1 \\ 1 & 0 \end{smallmatrix} \right)
\right\}$.
Given
$M=\left( \begin{smallmatrix} i & j \\ k & \ell \end{smallmatrix} \right)
\in \bM$, define the morphism $\gamma_M : \aff^2 \to \aff^2$ by
$(x,y) \mapsto (x^iy^j, x^ky^\ell)$.
Note that $\gamma_{M_1} \circ \gamma_{M_2} = \gamma_{M_1M_2}$
for all $M_1, M_2 \in \bM$, so the set
$$
\bG = \bG_\cgoth =\setspec{ \gamma_M }{ M \in \bM }
$$
is a monoid (under composition) generated by
$\left\{
\gamma_{ \left( \begin{smallmatrix} 1 & 1 \\ 0 & 1 \end{smallmatrix} \right) },
\gamma_{ \left( \begin{smallmatrix} 0 & 1 \\ 1 & 0 \end{smallmatrix} \right) }
\right\}$.
As $\gamma_{\left(\begin{smallmatrix} 1 & 1 \\ 0 & 1 \end{smallmatrix}\right)}$
is a SAC and
$\gamma_{ \left( \begin{smallmatrix} 0 & 1 \\ 1 & 0 \end{smallmatrix} \right)}$
is an automorphism, it follows that $\bG$ is a submonoid of $\BirA$.
\end{subparag}

\begin{subparag} \label{2376123y234ygd76}
Given a polynomial $\phi \in \bk[X]\setminus \{0\}$,
define $v_\phi \in \BirA$ by $v_\phi(x,y) = (x, \phi(x)y)$.
Then let 
$$
\bV = \bV_\cgoth = \setspec{ v_\phi }{ \phi \in  \bk[X]\setminus \{0\} } .
$$
Note that $v_\phi \circ v_\psi = v_{\phi \cdot \psi} = v_\psi \circ v_\phi$
for any $\phi,\psi \in \bk[X] \setminus \{0\}$, so $\bV$ is a submonoid of $\BirA$.  
\end{subparag}

\begin{lemma}  \label{dfpiwer;lkasd}
For a birational morphism $f : \aff^2 \to \aff^2$, the following are
equivalent:
\begin{enumerata}

\item The union of the missing curves of $f$ is included in $\bZ(Y)$;

\item there exists $(h,\theta) \in \bH \times \Aut(\aff^2)$
such that $f=h \circ \theta$:

\end{enumerata}
$$
\xymatrix{
{\aff^2} \ar[r]_{ \theta } \ar @/^1pc/ [rr]^{f}  
&   {\aff^2}  \ar[r]_{ h }  &   {\aff^2}
} 
$$
Moreover, if these conditions are satisfied then the pair
$(h,\theta)$ in \mbox{\rm (b)} is unique.
\end{lemma}

\begin{proof}
We leave it to the reader to verify that (b) implies (a) and that 
$(h,\theta)$ is unique, in (b).
By induction on $n(f)$, we show that (a) implies (b).

If $n(f)=0$ then (b) holds with $\theta=f$ and $h = h_{0,0}$.

If $n(f)>0$ then $f$ is not an isomorphism, and hence has at least one
missing curve; so $\bZ(Y)$ is the unique missing curve of $f$;
by \ref{dkfjaiuer834134k1j3}, this missing curve is blown-up only once.
This missing curve must contain a fundamental point $(c,0)$ of $f$;
as $h_{1,c} \in \bH$ is a SAC with missing curve $\bZ(Y)$
and fundamental point $(c,0)$, \ref{dkfjaiuer834134k1j3} implies that
$f = h_{1,c}\circ f'$ for some birational morphism $f' : \aff^2 \to \aff^2$.
It is immediate that $h_{1,c}^{-1}(\Gamma)=\Gamma$,
where $\Gamma = \bZ(Y)$ is the
missing curve of $f$; so \ref{Doifqrqkkdsjf8} implies that 
the union of the missing curves of $f'$ is included in $\bZ(Y)$.
As $n( f' ) = n(f)-1$, we may assume by induction that
$f' = h' \circ \theta$ for some $h' \in \bH$ and $\theta \in \Aut(\aff^2)$.
Then $f = h_{1,c} \circ h' \circ \theta$ is the desired factorization,
where we note that $h_{1,c} \circ h' \in \bH$.
\end{proof}

\begin{lemma} \label{dfp198234okjdao}
For a birational morphism $f : \aff^2 \to \aff^2$, the following are
equivalent:
\begin{enumerata}

\item The union of the missing curves of $f$ is included in $\bZ(XY)$;

\item there exists $(M,h,\theta) \in \bM \times \bH \times \AutA$
such that $f=\gamma_M \circ h \circ \theta$:

\end{enumerata}
$$
\xymatrix{
{\aff^2} \ar[r]_{ \theta } \ar @/^1pc/ [rrr]^{f}  
&   {\aff^2}  \ar[r]_{ h }  &   {\aff^2} \ar[r]_{ \gamma_M }  &   {\aff^2}
} 
$$
\end{lemma}

\begin{proof}
It is easily verified that (b) implies (a).
We prove that (a) implies (b) by induction on $n(f)$.
Assume that $f$ satisfies (a).

If $n(f) = 0$ then $f$ is an isomorphism, so (b) holds
with $\theta=f$, $h=h_{0,0}$ and
$M=\left(\begin{smallmatrix} 1&0 \\ 0&1 \end{smallmatrix} \right)$.

Let $n>0$ and assume that the result is true whenever $n(f)<n$.
Now consider $f$ satisfying (a) and such that $n(f)=n$.

If $q(f)=1$ then the missing curve $\Gamma$ of $f$ is $\bZ(X)$ or $\bZ(Y)$.
Define $M = 
\left(\begin{smallmatrix} 1&0 \\ 0&1 \end{smallmatrix} \right)$
(resp.\ $M = 
\left(\begin{smallmatrix} 0&1 \\ 1&0 \end{smallmatrix} \right)$)
if $\Gamma = \bZ(Y)$ (resp.\ $\Gamma = \bZ(X)$).
Then $\gamma_{M}\circ f$ has a unique missing curve, and this
curve is $\bZ(Y)$. Applying \ref{dfpiwer;lkasd} to 
$\gamma_{M}\circ f$ gives 
$\gamma_{M}\circ f = h \circ \theta$ for some $\theta \in \AutA$
and $h \in \bH$.  Noting that 
$\gamma_{M}\circ \gamma_{M}$ is the identity, we get
$f = \gamma_{M} \circ h \circ \theta$.

From now-on,  assume that $q(f)=2$.
Let $\Gamma$ be the union of the missing curves of $f$, i.e., 
$\Gamma=\bZ(XY)$.
By \ref{dkfjaiuer834134k1j3}, some element $L$ of 
$\Miss(f) = \big\{ \bZ(X),  \bZ(Y) \big\}$ is blown-up only once.
As $(0,0)$ is a common point of the two missing curves, it must be 
a fundamental point of $f$. For a suitable choice of 
$M_1 \in \left\{
\left( \begin{smallmatrix} 1 & 1 \\ 0 & 1 \end{smallmatrix} \right) ,
\left( \begin{smallmatrix} 1 & 0 \\ 1 & 1 \end{smallmatrix} \right)
\right\}$,
$\gamma_{M_1}$ is a SAC with missing curve $L$ and fundamental
point $(0,0)$. Then \ref{dkfjaiuer834134k1j3} implies that
$f=\gamma_{M_1} \circ f'$ for some 
birational morphism $f' : \aff^2 \to \aff^2$.
By \ref{Doifqrqkkdsjf8},
the union of the missing curves of $f'$ is included in
$\gamma_{M_1}^{-1}( \Gamma ) = \Gamma$, so $f'$ satisfies (a).
As $n(f') = n(f)-1$, the inductive hypothesis implies that
$f' = \gamma_{M_2} \circ h \circ \theta$ for some 
$\theta\in\AutA$, $h \in \bH$ and $M_2 \in \bM$.
So $f= \gamma_{M_1} \circ \gamma_{M_2} \circ h \circ \theta$,
and since $\gamma_{M_1} \circ \gamma_{M_2} = \gamma_{M_1M_2}$, 
we are done.
\end{proof}

\begin{parag}
Let $\Delta = \Delta_\cgoth$ be the subgroup of $\AutA$ whose elements
are of the form $\delta(x,y) = (x,y+q(x))$, with $q \in \bk[X]$.
\end{parag}

\begin{lemma} \label {ffj388383jrwejxk}
Let $\Gamma = \bZ\big( \prod_{i=1}^s (X-c_i) \big)$ where
$c_1, \dots, c_s$ ($s\ge0$) are distinct elements of  $\bk$.
For a birational morphism $f : \aff^2 \to \aff^2$, the following are
equivalent:
\begin{enumerata}

\item The union of the missing curves of $f$ is included in $\Gamma$;

\item there exists a commutative diagram
$$
\xymatrix{
{\aff^2} \ar[r]^{f} \ar[d]_{\theta} & {\aff^2} \ar[d]^{\delta}  \\
{\aff^2}  \ar[r]_{ v_\phi } & {\aff^2}
}
$$
where $\delta \in \Delta$, $\theta \in \AutA$,
$\phi \in \bk[X]\setminus\{0\}$ and where the set of roots of $\phi$ is
included in $\{c_1, \dots, c_s\}$.

\end{enumerata}
\end{lemma}

\begin{proof}
That (b) implies (a) is left to the reader.
Suppose that $f$ satisfies (a).  We prove (b) by induction on $n(f)$.

If $n(f)=0$ then $f$ is an isomorphism, so (b) holds with $\theta=f$,
$\phi=1$ and $\delta=\id$.

Let $n>0$ be such that the result is true whenever $n(f)<n$.
Consider $f$ satisfying (a) and such that $n(f)=n$.
Then $f$ is not an isomorphism, and hence has at least one missing
curve (so $s > 0$).
By \ref{dkfjaiuer834134k1j3}, one of the missing curves
(say $L=\bZ(X-c_j)$) of $f$ is blown-up only once.
We know that $L$ contains a fundamental point $(c_j,d)$ of $f$;
let $\delta_1 \in \Delta$ be defined by $\delta_1(x,y) = (x,y-d)$
and let $f_1 = \delta_1 \circ f$.
Since $L$ is a missing curve of $f$ which is blown-up only once and 
$(c_j,d) \in L$ is a fundamental point of $f$, it follows that 
$\delta_1(L)=L$ is a missing curve of $f_1$ which is blown-up only once
and that $\delta_1(c_j,d)=(c_j,0) \in L$ is a fundamental point of $f_1$.
As $v_{(X-c_j)}$ is a SAC with missing curve $L$ and fundamental point
$(c_j,0)$, \ref{dkfjaiuer834134k1j3} implies that $f_1$ factors through 
$v_{(X-c_j)}$. Thus 
$\delta_1 \circ f = v_{(X-c_j)} \circ f'$ for some birational morphism
$f' : \aff^2 \to \aff^2$:
$$
\xymatrix{
{\aff^2} \ar[r]^{f} \ar[d]_{f'} & {\aff^2} \ar[d]^{\delta_1}  \\
{\aff^2}  \ar[r]_{ v_{(X-c_j)} } & {\aff^2}
}
$$
Since $\delta_1$ maps each vertical line onto itself, the union of all missing curves of $f_1$
is $\Gamma$; so, by \ref{Doifqrqkkdsjf8},
the union of the missing curves of $f'$ is included in
$v_{(X-c_j)}^{-1}( \Gamma ) = \Gamma$, so $f'$ satisfies (a).
As $n(f') = n(f)-1$, the inductive hypothesis implies that
there exists a commutative diagram (ignore the dotted arrows for now)
\begin{equation}  \label{diufq23uwe9jds}
\xymatrix{
& {\aff^2} \ar[r]^{f} \ar[d]_{f'} \ar[dl]_{\theta}
& {\aff^2} \ar[d]^{\delta_1}  \\
{\aff^2} \ar[d]_{v_{\phi'}}
& {\aff^2}  \ar[r]_{ v_{(X-c_j)} }  \ar[dl]_{\delta_2}
& {\aff^2}  \ar @{.>}  @<.8ex> [dl]^{\delta_3} \\
{\aff^2} \ar @{.>} [r]_{v_{(X-c_j)}} & {\aff^2}
}
\end{equation}
with $\phi' \in \bk[X]\setminus \{0\}$ (and all roots of $\phi'$
are in $\{c_1, \dots, c_s \}$),
$\theta \in \AutA$,
and $\delta_2 \in \Delta$ defined by
$\delta_2(x,y)=(x,y+q(x))$ (some $q \in \bk[X]$).
Now if we define $\delta_3 \in \Delta$
by $\delta_3(x,y)=(x,y+(x-c_j)q(x))$, then 
$$
\delta_3 \circ v_{(X-c_j)} = v_{(X-c_j)} \circ \delta_2.
$$
So diagram \eqref{diufq23uwe9jds}, including the dotted arrows,
is commutative.
Let $\delta = \delta_3 \circ \delta_1 \in \Delta$ and 
$\phi = (X-c_j)\phi'(X)$ (so $v_{(X-c_j)}\circ v_{\phi'} = v_\phi$);
then $\delta, \theta, v_\phi$ give the commutative diagram displayed
in the statement of assertion~(b).
\end{proof}

\begin{lemma} \label {seur823412y8dshfjs}
Let $\Gamma = \bZ\big( Y \prod_{i=1}^s (X-c_i) \big)$,
where $s\ge1$ and $c_1, \dots, c_s$ are distinct elements of $\bk$.
Let $f : \aff^2 \to \aff^2$ be a birational morphism such that:
\begin{equation} \label{dsjfq2rewkdfisdkddddd}
\textit{The union of the missing curves of $f$ is equal to $\Gamma$.}
\end{equation}
Then there exists a commutative diagram
$$
\xymatrix{
{\aff^2} \ar[rrr]^{f} \ar[d]_{\theta} &&& {\aff^2} \ar[d]^{T}  \\
{\aff^2} \ar[r]_{ h } & {\aff^2}  \ar[r]_{ \gamma_M }
& {\aff^2}  \ar[r]_{ v_\phi } & {\aff^2}
}
$$
where $T \in \AutA$ is of the form $T(x,y)=(x-c,y)$ with $c \in \bk$,
$\theta$ is an arbitrary element of $\AutA$,
and $(\phi,M,h) \in (\bk[X]\setminus\{0\}) \times \bM \times \bH$.
\end{lemma}

\begin{proof}
We first settle the case $s=1$.
Define $T \in \AutA$ by $T(x,y)= (x-c_1,y)$.
Then the union of the missing curves of $T \circ f$ is $\bZ(XY)$,
so \ref{dfp198234okjdao} implies that there
exists $(M,h,\theta) \in \bM \times \bH \times \AutA$
such that $T \circ f =\gamma_M \circ h \circ \theta
=v_\phi \circ \gamma_M \circ h \circ \theta$ with $\phi=1$,
so the result is true when $s=1$.

We proceed by induction on $n(f)$.
For $f$ satisfying \eqref{dsjfq2rewkdfisdkddddd} we have $q(f)=s+1\ge2$,
so the least possible value for $n(f)$ is $2$.
If $n(f)=2$ then $q(f)\le n(f)=2$, so $s=1$ and the result is true in this case.

Let $n>2$ be such that the result is true whenever $n(f)<n$.
Consider $f$ satisfying \eqref{dsjfq2rewkdfisdkddddd} and such that $n(f)=n$.

By the first paragraph, we may assume that $s>1$.
By \ref{dkfjaiuer834134k1j3}, one of the missing curves
(say $L$) of $f$ is blown-up only once; we choose such an $L$.
By \ref{dijfaksdjfall00000}\eqref{000006}, the points
$(c_i,0)$, $1\le i\le s$, are fundamental points of $f$;
so $\bZ(Y)$ is blown-up at least $s\ge2$
times and hence $L=\bZ(X-c_j)$ for some $j \in \{1, \dots, s\}$.
As $v_{(X-c_j)}$ is a SAC with missing curve $L$ and fundamental
point $(c_j,0)$, 
\ref{dkfjaiuer834134k1j3} implies that $f = v_{(X-c_j)} \circ f'$
for some birational morphism $f' : \aff^2 \to \aff^2$.
Let $\Gamma' = \bigcup_{C \in \Miss(f')}C$.
By \ref{Doifqrqkkdsjf8},
$\Gamma' \subseteq v_{(X-c_j)}^{-1}( \Gamma ) = \Gamma$;
in fact it is easy to see (again by \ref{Doifqrqkkdsjf8}) that
$$
\bZ\big( Y \prod_{i \in I} (X-c_i) \big)
\subseteq \Gamma' \subseteq
\bZ\big( Y \prod_{i=1}^s (X-c_i) \big),
$$
where $I = \{1, \dots, s\} \setminus \{j\}$,
so $f'$ satisfies the hypothesis of the lemma.
As $n(f')=n(f)-1$, the inductive hypothesis implies that
$T \circ f' = v_\psi \circ \gamma_M \circ h \circ \theta$ 
for some $(\psi,M,h) \in (\bk[X]\setminus\{0\}) \times \bM \times \bH$
and $\theta, T \in \AutA$,
where $T$ is of the form $T(x,y)=(x-c,y)$ for some $c \in \bk$.
Noting that $ T \circ v_{(X-c_j)} = v_{(X+c-c_j)} \circ T $,
we get
\begin{multline*}
T \circ f = T \circ v_{(X-c_j)} \circ f'
= v_{(X+c-c_j)} \circ T \circ f'
= v_{(X+c-c_j)} \circ  v_\psi \circ \gamma_M \circ h \circ \theta \\
= v_\phi \circ \gamma_M \circ h \circ \theta,
\end{multline*}
where $\phi(X) = (X+c-c_j)\psi(X) \in \bk[X] \setminus \{0\}$,
as desired. 
\end{proof}

Before stating the main theorem of this section, let us recall the assumptions under which it is valid.
Our base field $\bk$ is an algebraically closed field of arbitrary characteristic, 
and $\aff^2$ is the affine plane over $\bk$.
We fix a coordinate system $\cgoth=(X,Y)$ of $\aff^2$;
this allows us to 
use coordinates for defining morphisms $\aff^2 \to \aff^2$ (cf.\ Section~\ref{Sec:AdmissConfig}).
The choice of $\cgoth$ also determines the submonoids 
$\bV = \bV_\cgoth$,  $\bG = \bG_\cgoth$ and  $\bH = \bH_\cgoth$ of $\Bir(\aff^2)$ (cf.\ \ref{27384ygfq7239j}).
Then we have the following result:

\begin{theorem}  \label {du-239412948ynf}
Let $f : \aff^2 \to \aff^2$ be a birational morphism.
\begin{enumerata}

\item The following conditions are equivalent: 

\begin{enumerata}

\item[(i)] $\Miss(f)$ is weakly admissible;

\item[(ii)] $f$ is equivalent to one of the following elements of $\BirA$:
\begin{itemize}

\item $\alpha_i^m \circ v \circ \gamma \circ h$,
for some $(v,\gamma,h) \in \bV \times \bG \times \bH$, $m \in \{0,1\}$ and $i \in \{ 1, 2 \}$,
where $\alpha_1,\alpha_2 \in \Bir(\aff^2)$ are the SACs defined by
$\alpha_1(x,y) = \big( xy, y \big)$ and $\alpha_2(x,y) = \big( x(1-y), 1-y \big)$;

\item the morphism $(x,y) \mapsto \big( x( p(x)y+q(x)), p(x)y+q(x) \big)$,
for some $p,q \in \bk[X]$ with $p \neq 0$.

\end{itemize}

\end{enumerata}

\item The following conditions are equivalent:

\begin{enumerata}

\item[(i)] $\Miss(f)$ is admissible;

\item[(ii)] $f$ is equivalent to $v \circ \gamma \circ h$ for some $(v,\gamma,h) \in \bV \times \bG \times \bH$.

\end{enumerata}

\item The following conditions are equivalent:

\begin{enumerata}

\item[(i)] Both $\Miss(f)$ and $\Cont(f)$ are admissible;

\item[(ii)] $f$ is equivalent to an element of $\bV \cup \bG$.

\end{enumerata}

\end{enumerata}
\end{theorem}

\begin{proof}
For each of (a), (b) and (c), we show that (i) implies (ii) and leave the converse to the reader.
We begin with (b).

Suppose that $f$ satisfies (b-i). 
Let $\Gamma = \bigcup_{C \in \Miss(f)}C$.
By \ref{dpfq32748rjkdkd}, there exists $\omega \in \AutA$
such that $\omega( \Gamma ) = \bZ\big( Y^d \prod_{i=1}^s (X-c_i) \big)$,
where:
$$
\text{$d \in \{0,1\}$, $s\ge0$ 
and $c_1, \dots, c_s$ are distinct elements of $\bk$.}
$$

Note that the union of the missing curves of $f_1 = \omega \circ f$
is equal to 
$\bZ\big( Y^d \prod_{i=1}^s (X-c_i) \big)$;
as $f_1\sim f$, it is enough to prove that
$f_1$ is equivalent to $v \circ \gamma \circ h$ for some $(v,\gamma,h) \in \bV \times \bG \times \bH$.
So we may as well replace $f$ by $f_1$ throughout; so from now-on we assume
that 
$$
\Gamma = \bigcup_{C \in \Miss(f)} \! C \,\,
= \bZ\big( Y^d \prod_{i=1}^s (X-c_i) \big).
$$

If $d=0$, (resp.\ $s=0$),
then the desired conclusion follows from \ref{ffj388383jrwejxk}
(resp.\ from \ref{dfpiwer;lkasd}).
So we may assume that $d=1$ and $s\ge1$.
Then \ref{seur823412y8dshfjs} gives the desired conclusion,
i.e., we showed that (b-i) implies (b-ii).

Suppose that (a-i) holds. 
Let $\Gamma = \bigcup_{C \in \Miss(f)}C$.
By \ref{weak-dpfq32748rjkdkd}, $f$ satisfies the hypothesis of \ref{dkfjaiuer834134k1j3}.
To prove (a-ii), we may assume that $\Miss(f)$ is not admissible (otherwise (a-ii) follows from (b)).
Then \ref{dkfjaiuer834134k1j3} implies that there exists $\omega \in \AutA$
such that $\omega(\Gamma) = \bZ(F)$ where 
\begin{equation} \label {86tg6g-----d33n}
F = Y \textstyle \prod_{i=1}^s (X-c_iY) \quad \text{or} \quad F = Y(Y-1) \textstyle \prod_{i=1}^s (X-c_iY),
\end{equation}
where $s\ge2$ and $c_1, \dots, c_s \in \bk$ are distinct.
We know, also by \ref{dkfjaiuer834134k1j3}, that some missing curve of $f$ (say $C_0 \in \Miss(f)$)
is blown-up only once.  In the second case of \eqref{86tg6g-----d33n},
$\omega(C_0)$ is necessarily equal to $\bZ(Y)$;
in the first case, we may choose $\omega$ in such a way that $\omega(C_0) = \bZ(Y)$.

It is clear that we may replace $f$ by $\omega \circ f$ throughout.
Then we have $\Gamma = \bZ(F)$,
$\bZ(Y)$ is a missing curve of $f$ which is blown-up only once,
and $(0,0)$ is the unique fundamental point of $f$ which lies on $\bZ(Y)$.
If $F$ is as in the first (resp.~the second) case of \eqref{86tg6g-----d33n}, let $\alpha=\alpha_1$
(resp.~ $\alpha=\alpha_2$), where $\alpha_1, \alpha_2 \in \BirA$ are defined in the statement;
then $\alpha$ is a SAC with missing curve $\bZ(Y)$ and fundamental point $(0,0)$.
By \ref{dkfjaiuer834134k1j3}, it follows that $f = \alpha \circ f'$
for some $f' \in \Bir(\aff^2)$.
Let $\Gamma'$ be the union of the missing curves of $f'$.
Using \ref{Doifqrqkkdsjf8}, we find 
\begin{align*}
\text{in the first case of \eqref{86tg6g-----d33n},}\quad
& \textstyle \bZ\big( \prod_{i =1}^s (X-c_i) \big)
\subseteq \Gamma' \subseteq
\bZ\big( Y \prod_{i=1}^s (X-c_i) \big);\\
\text{in the second case of \eqref{86tg6g-----d33n},}\quad
& \textstyle \bZ\big( Y \prod_{i =1}^s (X-c_i) \big)
\subseteq \Gamma' \subseteq
\bZ\big( Y(Y-1) \prod_{i=1}^s (X-c_i) \big).
\end{align*}
In particular, $f'$ satisfies the hypothesis of \ref{dkfjaiuer834134k1j3};
by that result, some missing curve of $f'$ is blown-up only once;
so $\Gamma'$ cannot be equal to $\bZ\big( Y(Y-1) \prod_{i=1}^s (X-c_i) \big)$.
It follows that $\Gamma' = \bZ(G)$ where
\begin{equation} \label {23dvcrk9iiiifda;}
G = \textstyle \prod_{i=1}^s (X-c_i) \quad \text{or} \quad G = Y\textstyle \prod_{i=1}^s (X-c_i).
\end{equation}
First consider the case $G = \textstyle \prod_{i=1}^s (X-c_i)$;
then $\alpha=\alpha_1$,
because the first case of \eqref{23dvcrk9iiiifda;} can only happen in the first case of \eqref{86tg6g-----d33n}.
By \ref{ffj388383jrwejxk}, there is a commutative diagram
$$
\xymatrix{
{\aff^2} \ar[r]^{f'} \ar[d]_{\theta} & {\aff^2} \ar[d]^{\delta}  \\
{\aff^2}  \ar[r]_{ v } & {\aff^2}
}
$$
where $v \in \bV$, 
$\delta, \theta \in \AutA$ and $\delta$ is of the form
form $\delta(x,y) = (x, y-q(x))$ for some $q \in \bk[X]$.  Then
$f = \alpha_1 \circ f'
= \alpha_1 \circ \delta^{-1} \circ v \circ \theta 
\sim \alpha_1 \circ \delta^{-1} \circ v$.
Let $p \in \bk[X] \setminus \{0\}$ be such that $v(x,y) = (x,p(x)y)$, then
$$
(\alpha_1 \circ \delta^{-1} \circ v)(x,y) = \big( x( p(x)y+q(x)), p(x)y+q(x) \big),
$$
which shows that (a-ii) holds in this case.

Consider the second case, $G = Y\textstyle \prod_{i=1}^s (X-c_i)$.
Here, $\alpha$ may be either one of $\alpha_1, \alpha_2$.
By \ref{seur823412y8dshfjs}, there is a commutative diagram
$$
\xymatrix{
{\aff^2} \ar[rr]^{f'} \ar[d]_{\theta} && {\aff^2} \ar[d]^{T}  \\
{\aff^2} \ar[rr]_{ v\circ \gamma \circ h } && {\aff^2}
}
$$
where $(v,\gamma,h) \in \bV \times \bG \times \bH$,
$\theta \in \AutA$,
and $T \in \AutA$ is of the form $T(x,y)=(x-c,y)$, with $c \in \bk$.
Now $\alpha \circ T^{-1} = \nu \circ \alpha$ where $\nu \in \AutA$ is given by $\nu(x,y)=(x+cy,y)$.
Thus
$$
f
= \alpha \circ f'
= \alpha \circ T^{-1} \circ v\circ \gamma \circ h  \circ \theta
= \nu \circ \alpha \circ v\circ \gamma \circ h  \circ \theta
\sim
\alpha \circ v\circ \gamma \circ h,
$$ 
showing that (a-ii) holds in this case as well.
So (a-i) implies (a-ii).

\medskip
Let $\phi \in \bk[X] \setminus \{0\}$, $M =
\left( \begin{smallmatrix} i & j \\ k & \ell \end{smallmatrix} \right) \in \bM$
and $h_{m,p}(x,y) = (xy^m+p(y), y) \in \bH$,
where $m \in \Nat$ and $p(Y) \in \bk[Y]$ are such that $\deg p < m$.
As a preparation for the proof that (c-i) implies (c-ii), we first show:
\begin{gather}
\label{FirstClaim}
\begin{minipage}[t]{.9\textwidth}
\it If $\Cont( \gamma_M \circ h_{m,p} )$ is admissible, then
$\gamma_M \circ h_{m,p} \sim \gamma$ for some $\gamma \in \bG$.
\end{minipage} \\
\label{SecondClaim}
\begin{minipage}[t]{.9\textwidth}
\it If $\Cont( v_\phi \circ h_{m,p} )$ is admissible, then
$v_\phi \circ h_{m,p}$ is equivalent to 
an element of $\bV \cup \bG$.
\end{minipage}
\end{gather}
Observe:
\begin{equation}  \label{iduf23u48iqwhd}
(\gamma_M \circ h_{m,p})(x,y) =
\big( (xy^m+p(y))^i y^j,\, (xy^m+p(y))^k y^\ell \big).
\end{equation}
To prove \eqref{FirstClaim}, first consider the case $ik\neq 0$;
then \eqref{iduf23u48iqwhd} implies that $\bZ( XY^m + p(Y) )$
is a contracting curve (or a union of contracting curves)
of $\gamma_M \circ h_{m,p}$.
So, by the hypothesis of \eqref{FirstClaim},
each irreducible component of $\bZ( XY^m + p(Y) )$ has one
place at infinity.  The only way to achieve this is to have $p=0$,
in which case we have $h_{m,p} = \gamma_{M'}$ with $M'=
\left( \begin{smallmatrix} 1 & m \\ 0 & 1 \end{smallmatrix} \right)$.
Then  $\gamma_M \circ h_{m,p} = \gamma_N \in \bG$ with $N=MM'$,
so \eqref{FirstClaim} is true in this case.

Consider next the case where $ik=0$. Then $M \in \big\{
\left( \begin{smallmatrix} 0 & 1 \\ 1 & \ell \end{smallmatrix} \right),
\,
\left( \begin{smallmatrix} 1 & j \\ 0 & 1 \end{smallmatrix} \right) \big\}$
for some $j,\ell \in \Nat$.
If 
$M = \left( \begin{smallmatrix} 0 & 1 \\ 1 & \ell \end{smallmatrix} \right)$,
then 
$ (\gamma_M \circ h_{m,p})(x,y) = \big( y,\, (xy^m+p(y)) y^\ell \big)
= \big( y,\, xy^{m+\ell} + y^\ell p(y) \big) $,
which is equivalent to the birational morphism $(x,y) \mapsto 
( y,\, xy^{m+\ell} ) $, i.e., 
$\gamma_M \circ h_{m,p} \sim \gamma_N \in \bG$ with 
$N = \left( \begin{smallmatrix} 0 & 1 \\ 1 & m+\ell \end{smallmatrix} \right)$.
Similarly, if 
$M=\left( \begin{smallmatrix} 1 & j \\ 0 & 1 \end{smallmatrix} \right)$
then 
$\gamma_M \circ h_{m,p} \sim \gamma_N \in \bG$ with 
$N = \left( \begin{smallmatrix} 1 & m+j \\ 0 & 1 \end{smallmatrix} \right)$.
This completes the proof of \eqref{FirstClaim}.

To prove \eqref{SecondClaim},
we first note that if $m=0$ then 
$v_\phi \circ h_{m,p} = v_\phi \circ \id = v_\phi \in \bV$.
Likewise, if $\phi \in \bk^*$ then $v_\phi$ is an isomorphism,
so $v_\phi \circ h_{m,p} \sim h_{m,p} \sim \gamma_
{\left( \begin{smallmatrix} 1 & m \\ 0 & 1 \end{smallmatrix} \right)} \in \bG$.
So we may assume from now-on that $m>0$ and that $\phi$ has at least one root.

If $c \in \bk$ is a root of $\phi$ then
$\bZ( XY^m + p(Y)-c )$ is a union of contracting curves of 
$v_\phi \circ h_{m,p}$,
so, by the hypothesis of \eqref{SecondClaim},
each irreducible component of $\bZ( XY^m + p(Y)-c )$
has one place at infinity. As $m>0$, this implies that $p(Y)-c$ is the
zero polynomial, and this is true for each root $c$ of $\phi$.
So $\phi=a(X-c)^n$ for some $a \in \bk^*$ and $n \ge 1$, and
$h_{m,p}(x,y)=( xy^m+c,y )$.
Then $(v_\phi \circ h_{m,p})(x,y) = ( xy^m+c, a(xy^m)^n y )$,
which is equivalent to $(x,y) \mapsto  ( xy^m, (xy^m)^n y )$,
i.e., 
$v_\phi \circ h_{m,p} \sim \gamma_N\in\bG$ with 
$N =
\left( \begin{smallmatrix} 1 & m \\ n & mn+1 \end{smallmatrix} \right)$.
This proves \eqref{SecondClaim}.

To prove that (c-i) implies (c-ii),
we consider $f= v_\phi \circ \gamma_M \circ h$
for some $(\phi,M,h) \in (\bk[X]\setminus\{0\}) \times \bM \times \bH$,
we assume that $\Cont(f)$ is admissible, and we have to prove (c-ii).
We use the notation 
$M = 
\left( \begin{smallmatrix} i & j \\ k & \ell \end{smallmatrix} \right) \in \bM$
and $h(x,y) = (xy^m+p(y), y)$ where $m \in \Nat$ and $p(Y) \in \bk[Y]$
are such that $\deg p < m$.

The assumption that $\Cont(f)$ is admissible implies in particular:
\begin{equation} \label{difjwqkefalksmdf,m}
\text{Each contracting curve of $v_\phi \circ \gamma_M$
has one place at infinity.}
\end{equation}
Indeed, suppose that
$C \in \Cont( v_\phi \circ \gamma_M )$ has more than one place at infinity;
then, by \ref{dijfaksdjfall00000}\eqref{000003},
$C$ is not a missing curve of $h$ and consequently
there exists a curve $C' \subset \aff^2$ such that
$h(C')$ is a dense subset of $C$.
Then $C'$ is a contracting curve of $f= v_\phi \circ \gamma_M \circ h$
but has more than one place at infinity
(because it dominates a curve with more than one place at infinity).
This contradicts the assumption that $\Cont(f)$ is admissible,
so \eqref{difjwqkefalksmdf,m} is proved.

We claim:
\begin{equation}  \label{dfjq34r7834823ehasdh}
\textit{$ij=0$ or $\phi(X)=aX^n$, for some $a \in \bk^*$ and $n\in\Nat$}.
\end{equation}
Indeed, suppose that $\phi$ is not of the form $aX^n$ with $a \in \bk^*$ and
$n \in \Nat$;  then there exists $c\in\bk^*$ such
that $\phi(c)=0$. Then $\bZ( x^iy^j - c )$ is a contracting curve
of $v_\phi \circ \gamma_M$ and, if $ij\neq 0$,
this curve has more than one place at infinity,
contradicting \eqref{difjwqkefalksmdf,m}.
So \eqref{dfjq34r7834823ehasdh} is proved.

Consider the case where $\phi(X)=aX^n$.
Then $v_\phi = \theta \circ \gamma_{M_1}$ where $\theta \in \AutA$
and $M_1=
\left( \begin{smallmatrix} 1 & 0 \\ n & 1 \end{smallmatrix} \right)
\in \bM$.
Then $f= v_\phi \circ \gamma_M \circ h
= \theta \circ \gamma_{M_1} \circ \gamma_M \circ h
\sim \gamma_{M_1M} \circ h$, so \eqref{FirstClaim} implies that
$f \sim \gamma$ for some $\gamma \in \bG$, so we are done in this case.

There remains the case $ij=0$; here we have
$M \in \big\{
\left( \begin{smallmatrix} 0 & 1 \\ 1 & \ell \end{smallmatrix} \right),
\,
\left( \begin{smallmatrix} 1 & 0 \\ k & 1 \end{smallmatrix} \right) \big\}$
for some $k,\ell \in \Nat$.

If $M=\left( \begin{smallmatrix} 1 & 0 \\ k & 1 \end{smallmatrix} \right)$
then $\gamma_M = v_{(X^k)}$ so $f=v_\phi \circ v_{(X^k)} \circ h_{m,p}
= v_{\phi_1} \circ h_{m,p}$ where $\phi_1 = X^k\phi(X)$,
so \eqref{SecondClaim} implies that 
$f$ is equivalent to an element of $\bV \cup \bG$ (so we are done).

If $M=\left( \begin{smallmatrix} 0 & 1 \\ 1 & \ell \end{smallmatrix} \right)$
then $M=M_1M_2$ where 
$M_1=\left( \begin{smallmatrix} 1 & 0 \\ \ell & 1 \end{smallmatrix} \right)$
and $M_2=
\left( \begin{smallmatrix} 0 & 1 \\ 1 & 0 \end{smallmatrix} \right)$.
Now $\gamma_{M_2} = \tau$, where $\tau \in \AutA$ is defined
by $\tau(x,y)=(y,x)$, and $\gamma_{M_1} = v_{(X^\ell)}$.
So we have 
$$
f \sim f \circ \tau
= v_\phi \circ \gamma_{M_1} \circ \gamma_{M_2} \circ h_{m,p} \circ \tau
= ( v_\phi \circ v_{(X^\ell)} ) \circ (\tau \circ h_{m,p} \circ \tau)
= v_{\phi_1} \circ (\tau \circ h_{m,p} \circ \tau),
$$
where $\phi_1 = X^\ell\phi(X)$. We have
$(\tau \circ h_{m,p} \circ \tau)(x,y) = (x,yx^m+p(x))$,
so 
$$
(v_{\phi_1} \circ (\tau \circ h_{m,p} \circ \tau))(x,y) = 
\big(x, \phi_1(x)(yx^m+p(x)) \big)
= \big(x, x^m\phi_1(x) y+\phi_1(x)p(x) \big),
$$
which is equivalent to the birational morphism $(x,y) \mapsto
\big(x, x^m\phi_1(x) y \big) = v_{\psi}(x,y)$ with $\psi = X^m\phi_1$.
So $f \sim v_{\psi} \in \bV$ and we have shown that (c-i) implies (c-ii).
\end{proof}

\begin{corollary}  \label {d;ksdjfp9-90`20}
Let $f \in \Bir \aff^2$.
Suppose that all missing curves of $f$ are lines, and that these are simultaneously rectifiable.
Then there exists a coordinate system of $\aff^2$ with respect to which
the configuration of missing curves is one of the following:
\setlength{\unitlength}{.75mm}
\begin{enumerata}
\setlength{\itemsep}{3mm}

\item 
\raisebox{-25\unitlength}{
\fbox{\begin{picture}(45,35)
\put(8,8){\line(0,1){24}}
\put(16,8){\line(0,1){24}}
\put(24,20){\makebox(0,0){\dots}}
\put(32,8){\line(0,1){24}}
\put(8,6){\makebox(0,0)[t]{\scriptsize $L_1$}}
\put(16,6){\makebox(0,0)[t]{\scriptsize $L_2$}}
\put(32,6){\makebox(0,0)[t]{\scriptsize $L_s$}}
\end{picture}}}
\quad\begin{minipage}{.6\textwidth}
Parallel lines $L_1, \dots, L_s$ ($s\ge0$).
\end{minipage}

\item
\raisebox{-25\unitlength}{
\fbox{\begin{picture}(45,35)
\put(8,8){\line(0,1){24}}
\put(16,8){\line(0,1){24}}
\put(24,23){\makebox(0,0){\dots}}
\put(32,8){\line(0,1){24}}
\put(4,15){\line(1,0){32}}
\put(8,6){\makebox(0,0)[t]{\scriptsize $L_1$}}
\put(16,6){\makebox(0,0)[t]{\scriptsize $L_2$}}
\put(32,6){\makebox(0,0)[t]{\scriptsize $L_s$}}
\put(37,15){\makebox(0,0)[l]{\scriptsize $L_{0}$}}
\end{picture}}}
\quad\begin{minipage}{.6\textwidth}
Parallel lines $L_1,\dots,L_s$ ($s\ge1$), plus one line $L_0$ not parallel to $L_1,\dots,L_s$.
\end{minipage}

\item
\raisebox{-27\unitlength}{
\fbox{\begin{picture}(45,37)(-20,-17)
\put(-15,0){\line(1,0){30}}
\put(-10,-15){\line(2,3){20}}
\put(10,-15){\line(-2,3){20}}
\put(0,10){\makebox(0,0){\dots}}
\put(16,0){\makebox(0,0)[l]{\scriptsize $L_1$}}
\put(10,16){\makebox(0,0)[bl]{\scriptsize $L_2$}}
\put(-10,16){\makebox(0,0)[br]{\scriptsize $L_s$}}
\end{picture}}}
\quad\begin{minipage}{.6\textwidth}
Concurrent lines $L_1, \dots, L_s$ ($s \ge3$).
\end{minipage}

\item
\raisebox{-27\unitlength}{
\fbox{\begin{picture}(45,37)(-20,-17)
\put(-10,-15){\line(2,3){20}}
\put(10,-15){\line(-2,3){20}}
\put(0,10){\makebox(0,0){\dots}}
\put(-15,0){\line(1,0){30}}
\put(-15,-8){\line(1,0){30}}
\put(16,-8){\makebox(0,0)[l]{\scriptsize $L_0$}}
\put(16,0){\makebox(0,0)[l]{\scriptsize $L_1$}}
\put(10,16){\makebox(0,0)[bl]{\scriptsize $L_2$}}
\put(-10,16){\makebox(0,0)[br]{\scriptsize $L_s$}}
\end{picture}}}
\quad\begin{minipage}{.6\textwidth}
Concurrent lines $L_1, \dots, L_s$ ($s \ge3$)
plus one line $L_0$, where $L_0$ is parallel to
one of the concurrent lines.
\end{minipage}
\end{enumerata}
Conversely, each of the above configurations of lines occurs as the configuration
of missing curves of some $f \in \BirA$.
\end{corollary}

The proof below gives, in each of the cases (a)--(d), an example of an $f \in \BirA$ having
the desired configuration of missing curves.

\begin{proof}[Proof of \ref{d;ksdjfp9-90`20}]
The hypothesis on $f$ is that $\Miss(f)$ is weakly admissible, so $f$ is described by
part \mbox{(a-ii)} of Theorem~\ref{du-239412948ynf}; it follows that $\Miss(f)$ must
be one of the configurations (a)--(d).
Note that $\Miss(f)$ is admissible in cases (a) and (b).
In cases (c) and (d), $\Miss(f)$ is weakly admissible but not admissible.

Conversely, consider the configurations of lines (a)--(d).
In each of the four cases we may 
choose a coordinate system $\cgoth = (X,Y)$ of $\aff^2$ with respect to which the configuration
of lines is $\bZ(F)$, where:
$$
F = \begin{cases}
\prod_{i=1}^s (X-c_i) & \text{in case (a),} \\
Y\prod_{i=1}^{s} (X-c_i) & \text{in case (b),} \\
Y\prod_{i=1}^{s-1} (X-c_iY) & \text{in case (c),} \\
Y(Y-1)\prod_{i=1}^{s-1} (X-c_iY) & \text{in case (d),}
\end{cases}
$$
where $c_1, \dots, c_s$ (resp.\ $c_1, \dots, c_{s-1}$) are distinct elements of $\bk$ in cases (a) and (b)
(resp.\ in cases (c) and (d)).
Let us exhibit, in each case, an $f \in \BirA$ such that the union of all missing curves of $f$ is $\bZ(F)$.
In cases (a) and (b), choose a univariate polynomial $\phi \in \bk[t]$ whose roots are exactly $c_1, \dots, c_s$,
and define $f \in \BirA$ by
$$
f(x,y) = \begin{cases}
(x, \phi(x)y), & \text{in case (a),} \\
(xy, \phi(xy)y), & \text{in case (b).} 
\end{cases}
$$
Then the union of the missing curves of $f$ is $\bZ(F)$, as desired.
In cases (c) and (d), first choose $g \in \BirA$ such that the union of the missing curves of $g$ is $\bZ(G)$,
where 
$$
G = \begin{cases}
\prod_{i=1}^{s-1} (X-c_i) & \text{in case (c),} \\
Y\prod_{i=1}^{s-1} (X-c_i) & \text{in case (d)}
\end{cases}
$$
(we know that $g$ exists by cases (a) and (b)).
Then define
$$
f = \begin{cases}
\alpha_1 \circ g & \text{in case (c),} \\
\alpha_2 \circ g  & \text{in case (d),} 
\end{cases}
$$
where $\alpha_1$ and $\alpha_2$ are defined in the statement of Theorem~\ref{du-239412948ynf}.
It follows from \ref{Doifqrqkkdsjf8}\eqref{qhv6ghqwka7} that the union of the missing curves of $f$ is $\bZ(F)$.
\end{proof}

\section{Some aspects of the monoid $\BirA$}
\label {SomeaspectsofthemonoidBirA}

Let $\bk$ be an algebraically closed field and $\aff^2=\aff^2_\bk$, and
consider the non commutative monoid $\BirA$ defined in the introduction.
Note that this is a cancellative monoid,
since it is included in the group of birational automorphisms of $\proj^2$.

In view of \ref{1dijfaksdjfall00100} and \ref{dijfaksdjfall00000}\eqref{000002},
it is clear that each non invertible element of $\BirA$ is a composition of finitely many irreducible elements.
In other words,
\begin{equation*}
\textit{the monoid $\BirA$ has factorizations into irreducibles.}
\end{equation*}
Essentially nothing is known regarding uniqueness of factorizations.\footnote{%
We do know that $\BirA$ is not a ``unique factorization monoid'' in the sense of \cite{John_uf-monoids71},
but this by no means settles the question of uniqueness of factorizations in $\BirA$.
Indeed, there are several non equivalent definitions of what one might mean by ``uniqueness of factorization'' in 
non commutative monoids, and the one used in  \cite{John_uf-monoids71} seems to be particularly inadequate in
the case of $\BirA$.}

It is natural to ask whether one can find all irreducible elements of $\BirA$ up to equivalence.
However, considering the examples given in 
\cite{Dai:bir}, \cite{Dai:trees} and \cite{CassouRussell:BirEnd}
and certain facts such as  \cite[4.12]{Dai:bir},
one gets the impression that the irreducible endomorphisms might be too numerous and too diverse to be listed.
The first part of the present section gives some simple observations
(\ref{.2938f9283w9fq0349hfio}--\ref{89j32467ewqfds}) that strengthen that impression.

\medskip
Given $f,g \in \BirA$, let us write $f \mid g$ if there exist $u,v \in \BirA$ such that $u\circ f \circ v = g$.
By a {\it prime element\/} of $\BirA$, we mean a non invertible element $p$ satisfying
$$
\text{for all $f,g \in \BirA$,\ \ $p \mid (g \circ f) \Rightarrow p \mid f \text{ or } p \mid g$.}
$$
It follows from \ref{1dijfaksdjfall00100} and \ref{dijfaksdjfall00000}\eqref{000002} that every
prime element of $\BirA$ is irreducible. 
It is natural to ask whether the converse is true, and in particular whether SACs are prime
(SACs are certainly irreducible).
These questions are open; we don't even know if there exists a prime element in $\BirA$.

\medskip
We say that a submonoid $\Meul$ of $\BirA$ is {\it factorially closed in $\BirA$} if
the conditions $f,g \in \BirA$ and $g \circ f \in \Meul$ imply $f,g \in \Meul$.
It is natural
to ask whether $\Aeul$ is factorially closed in $\BirA$,
where $\Aeul$ is the submonoid of $\BirA$ generated by SACs and automorphisms.\footnote{The question is natural
in view of the question whether SACs are prime and in view of the following trivial fact:
let $P$ be a set of prime elements in a commutative and cancellative monoid $\Neul$,
and let $\overline P$ be the submonoid of $\Neul$
generated by $P$ and all invertible elements of $\Neul$;
then $\overline P$ is factorially closed in $\Neul$.}
The main result of this section, Theorem~\ref{9kjxt12544d2jhdbhdfa384732}, states that
$\Aeul$ is indeed factorially closed in $\BirA$.

\begin{smallremark}
It is obvious that the only irreducible elements of $\Aeul$ are the SACs,
that each non invertible element of $\Aeul$ is a composition of irreducible elements,
and that $\Aeul$ has the following ``unique factorization'' property:
if $x_1, \dots, x_m, y_1, \dots, y_n$ are irreducible elements of $\Aeul$ such that
$x_1 \circ \cdots \circ x_m = y_1 \circ \cdots \circ y_n$, then $m=n$ and for each $i = 1, \dots, n$
we have $x_i = u_i \circ y_i \circ v_i$ for some invertible elements $u_i, v_i \in \Aeul$.
(However, it is easy to see that $\Aeul$ is not a unique factorization monoid in the sense
defined in \cite{John_uf-monoids71}.)
\end{smallremark}

\section*{Irreducible elements and generating sets}

We write $[f]$ for the equivalence class of an element $f$ of $\BirA$.

\begin{lemma} \label {.2938f9283w9fq0349hfio}
$ \big| \setspec{ [f] }{ \text{$f$ is an irreducible element of $\BirA$} } \big| = |\bk|$. 
\end{lemma}

\begin{proof}
Fix a coordinate system $(X,Y)$ of $\aff^2$. 
For each $a \in \bk^*$, let $C_a \subset \aff^2$ be the zero-set of\ \ $aY^2(Y-1)+X \in \bk[X,Y]$.

Define
$U = \setspec{(a_1,a_2,a_3)\in\bk^3}{ \text{$a_1,a_2,a_3$ are distinct and nonzero} }$.
Define an equivalence relation $\approx$ on the set $U$ by declaring that 
$(a_1,a_2,a_3) \approx (b_1,b_2,b_3)$ iff there exists $\theta \in \AutA$ satisfying
$\theta( C_{a_1} \cup  C_{a_2} \cup  C_{a_3} ) = C_{b_1} \cup  C_{b_2} \cup  C_{b_3}$.
The reader may check\footnote{This is a tedious exercise. We leave it to the reader because it is
completely elementary and has nothing to do with the subject matter of this paper.}
that the set $U/\!\approx$ of equivalence classes has cardinality $|\bk|$.

Given $q\ge2$ and distinct elements $a_1, \dots, a_q \in \bk^*$, there exists 
an irreducible element $f \in \BirA$ 
such that $\Miss(f) = \{ C_{a_1}, \dots, C_{a_q} \}$ and $n(f) = q+2$
(to see this, set $m=3$ and $\delta_1 = \cdots = \delta_{q-1}=0$ in \cite[4.13]{Dai:bir}).\footnote{Note that 
in Example 4.13 of \cite{Dai:bir} one has $\Miss(f) = \{C_1, \dots, C_q\}$.
This doesn't seem to be stated explicitly, but it is clear if one reads the construction.}
In particular,
for each $\mathbf{a} = (a_1,a_2,a_3) \in U$ there exists 
an irreducible $f_\mathbf{a} \in \BirA$
such that $\Miss(f_\mathbf{a}) = \{ C_{a_1}, C_{a_2}, C_{a_3} \}$.
If $\mathbf{a},\mathbf{b} \in U$
are such that $f_\mathbf{a} \sim f_\mathbf{b}$ then there exist $\theta,\theta' \in \AutA$ satisfying
$\theta \circ f_\mathbf{a} = f_\mathbf{b} \circ \theta'$;  then 
$\theta( C_{a_1} \cup  C_{a_2} \cup  C_{a_3} ) = C_{b_1} \cup  C_{b_2} \cup  C_{b_3}$,
so $\mathbf{a} \approx \mathbf{b}$.
By the preceding paragraph we get $| \setspec{ [f_\mathbf{a}] }{ \mathbf{a} \in U } | = |\bk|$,
from which the desired conclusion follows.
\end{proof}

\begin{lemma}  \label {92389h32fhfqfqasde}
For any subset $S$ of $\BirA$, the following are equivalent:
\begin{enumerate}

\item[(i)] $\AutA \cup S$ is a generating set for the monoid $\BirA$;
\item[(ii)] for each irreducible $f \in \BirA$, $[f] \cap S \neq \emptyset$.
\end{enumerate}
\end{lemma}

\begin{proof}
Suppose that $S$ satisfies (i) and consider an irreducible $f \in \BirA$. By (i), 
$$
f = g_1 \circ \cdots \circ g_n \text{\ \ for some finite subset $\{g_1, \dots, g_n \}$ of $\AutA \cup S$}.
$$
By irreducibility of $f$, exactly one element $g_i$ of $\{g_1, \dots, g_n \}$  is not in $\AutA$
(consequently, $g_i \in S$).
So $f \sim g_i \in S$, which proves that $S$ satisfies (ii).

Conversely, suppose that (ii) holds and consider $h \in \BirA$;  we claim that 
$$
h = g_1 \circ \cdots \circ g_N \text{\ \ for some finite subset $\{g_1, \dots, g_N \}$ of $\AutA \cup S$}.
$$
This is clear if $h \in \AutA$, so assume that $h \notin \AutA$.
Then $h = f_1 \circ \dots \circ f_n$ for some finite collection $\{ f_1, \dots, f_n \}$ of irreducible elements
of $\BirA$ (existence of a factorization into irreducibles is a consequence of \ref{1dijfaksdjfall00100}).
For each $i \in \{1, \dots, n\}$, we have $[f_i] \cap S \neq \emptyset$, so
$f_i = u_i \circ s_i \circ v_i$ for some $s_i \in S$ and $u_i,v_i \in \AutA$. Then
$$
h = (u_1 \circ s_1 \circ v_1) \circ  \cdots \circ (u_n \circ s_n \circ v_n) = g_1 \circ \cdots \circ g_N
$$
where $\{g_1, \dots, g_N \} \subset \AutA \cup S$. This proves (i).
\end{proof}

\begin{corollary} \label {2938923d89h209jdhvas.}
Let $S$ be a subset of $\BirA$ such that $\AutA \cup S$ is a generating set for the monoid $\BirA$.
Then $|S| = |\bk|$.
\end{corollary}

\begin{proof}
Follows from  \ref{.2938f9283w9fq0349hfio} and \ref{92389h32fhfqfqasde}.
\end{proof}

\begin{remark}
Let $f \in \BirA$ and let $\gamma=(X,Y)$ be a coordinate system of $\aff^2$.
Then $f : \aff^2 \to \aff^2$ is given by $f(x,y) = (u(x,y), v(x,y))$ for some polynomials $u,v \in \bk[X,Y]$.
We define $\deg_\gamma f = \max( \deg_\gamma u,  \deg_\gamma v )$. 
We may also define $\deg f$ to be the minimum of $\deg_\gamma f$ for $\gamma$ ranging over the set of 
coordinate systems of $\aff^2$. 
Then 
\begin{equation} \label {73ffj918}
\deg f \, \ge \, \frac{ c(f) + 2 }{2}.
\end{equation}
Indeed, if $F_1, \dots, F_c \in \bk[X,Y]$ are irreducible polynomials whose zero-sets are the 
contracting curves of $f$ (so $c(f)=c$) then
the jacobian determinant of $(u,v)$ with respect to $(X,Y)$ is divisible by $\prod_{i=1}^c F_i$.
This implies that $\deg_\gamma f \ge (c+2)/2$, where the right hand side is independent from $\gamma$.
Statement \eqref{73ffj918} follows.
\end{remark}

\begin{corollary}\footnote{This result answers a question posed by Patrick Popescu-Pampu.}
\label {89j32467ewqfds}
Let $S$ be a subset of $\BirA$ such that $\AutA \cup S$ is a generating set for the monoid $\BirA$.
Then $\setspec{ \deg f }{ f \in S }$ is not bounded.
\end{corollary}

\begin{proof}
Let $n \in \Nat$.
By \cite[4.13]{Dai:bir}, there exists an irreducible element $g \in \BirA$ satisfying $c(g) \ge 2n$.
By \ref{92389h32fhfqfqasde}, there exists $f \in S$ satisfying $f \sim g$; then $c(f) = c(g) \ge 2n$,
so $\deg f > n$ by \eqref{73ffj918}.
\end{proof}

\section*{Factorial closedness of $\Aeul$ in $\BirA$}

Let $\Aeul$ be the submonoid of $\BirA$ generated by SACs and automorphisms.

\smallskip
See \ref{9f2039esdow9e} for the definition of $n(f,C)$, where $f \in \BirA$ and $C \in \Miss(f)$.

\begin{lemma}  \label {092390hr3tgc}
Consider $\aff^2 \xrightarrow{\alpha} \aff^2 \xrightarrow{f} \aff^2$
where $\alpha,f \in \BirA$ and $\alpha$ is a SAC.
Assume that the missing curve $C$ of $\alpha$ is disjoint from $\exc(f)$
and let $D$ be the closure of $f(C)$ in $\aff^2$. 
Then there exist a SAC $\alpha'$ and some $f' \in \BirA$ satisfying
$f \circ \alpha = \alpha' \circ f'$ and $\Miss(\alpha') = \{D\}$.
Moreover, if $f$ is a SAC then so is $f'$.
\end{lemma}

\begin{proof}
By \ref{90f20rdw9019jds}\eqref{lwoj837456djd93},
we have $D \in \Miss( f \circ \alpha)$ and $n( f \circ \alpha , D )=n(\alpha,C)=1$
(because $C\cap \exc(f)=\emptyset$ and $C \isom \aff^1$).
Let $P$ be the unique fundamental point of $f \circ \alpha$ which lies on $D$
and let $\alpha'$ be a SAC with missing curve $D$ and fundamental point $P$. 
Then \ref{nouveau-dkfjaiuer834134k1j3}\eqref{8734yt9873y4reyw} implies that
$f \circ \alpha = \alpha' \circ f'$ for some $f' \in \BirA$.
Then $n(f')=n(f)$, so if $f$ is a SAC then so is $f'$.
\end{proof}

\begin{definition}
Let $h \in \BirA$ be such that $h \notin \AutA$.  Let $C \in \Miss(h)$.
\begin{enumerata}
\setlength{\itemsep}{1mm}

\item   A {\it factorization\/} of $h$ is
a tuple $\fgoth = (h_1, \dots, h_n)$ of elements of $\BirA$ satisfying
$h = h_1 \circ \cdots \circ h_n$ (where $n\ge1$).
If $h_1, \dots, h_n$ are SACs, we say that $\fgoth$ is a {\it factorization of $h$ into SACs.}

\item Given a factorisation $\fgoth = (h_1, \dots, h_n)$ of $h$,
we define $\depth_\fgoth(h,C)$ to be the unique $i \in \{1,\dots,n\}$ satisfying
$$
\begin{minipage}{.9\linewidth}
\it there exists a missing curve of $h_i$ whose image by 
$h_1 \circ \cdots \circ h_{i-1}$ is a dense subset of $C$.
\end{minipage}
$$
Observe that
$\depth_\fgoth(h,C)\ge1$ and that 
$\depth_\fgoth(h,C)=1 \iff C \in \Miss(h_1)$.

\item If $h \in \Aeul$ then we define
$$
\depth(h,C) = \min \setspec{ \depth_\fgoth (h,C) }{ \text{$\fgoth$ is a factorization of $h$ into SACs} }.
$$
Note that $\depth(h,C) \ge 1$, and that $\depth(h,C)=1$ is equivalent to the existence of SACs
$\alpha_1, \dots, \alpha_n$ satisfying
$$
h = \alpha_1 \circ \cdots \circ \alpha_{n} \quad \text{and} \quad \Miss(\alpha_1) = \{C\}.
$$

\end{enumerata}
\end{definition}

\begin{theorem}  \label {9kjxt12544d2jhdbhdfa384732}
If $f,g \in \BirA$ satisfy $g\circ f \in \Aeul$, then $f,g \in \Aeul$.
\end{theorem}

\begin{proof}
We proceed by induction on  $n(g \circ f)$, the result being trivial for $n(g \circ f) \le 2$.
Let $n \ge3$ be such that 
\begin{equation}  \tag{$*$}
\forall\, f,g \in \BirA,\qquad  g \circ f \in \Aeul \text{\ and\ } n(g \circ f)<n\ \implies\ f,g \in \Aeul.
\end{equation}
Consider $f,g \in \BirA$ such that $g \circ f \in \Aeul$ and $n(g \circ f)=n$; we have to show that $f,g \in \Aeul$.
Since $g \circ f \in \Aeul$, the number $\depth(g \circ f, C)$ is defined for every $C \in \Miss(g\circ f)$.
Observe that 
\begin{equation}  \label {8234hd1grthdenv}
\begin{minipage}[t]{.9\textwidth}
there exists $C \in \Miss(g \circ f)$ satisfying $\depth(g \circ f, C)=1$,
and any such $C$ satisfies $n(g \circ f,C)=1$.
\end{minipage}
\end{equation}
Indeed, for any factorization $g \circ f = \alpha_1 \circ \cdots \circ \alpha_{n}$ of $g \circ f$ into SACs,
the missing curve $C$ of $\alpha_1$ satisfies
$C \in \Miss(g \circ f)$ and $\depth(g \circ f, C)=1$, so $C$ exists.
Given any $C \in \Miss(g \circ f)$ satisfying $\depth(g \circ f, C)=1$,
there exists a factorization $g \circ f = \alpha_1 \circ \cdots \circ \alpha_{n}$ of $g \circ f$ into SACs
satisfying $\Miss(\alpha_1)=\{C\}$; then
$n(g\circ f, C) = n( \alpha_1 \circ \cdots \circ \alpha_{n} , C) = n(\alpha_1,C)=1$,
where the second equality follows from \ref{90f20rdw9019jds}\eqref{70209ru29fyt2suddh29}.
This proves \eqref{8234hd1grthdenv}.

We now proceed to prove that $f,g \in \Aeul$. 
We first do so in two special cases (numbered 1 and 2) and then in the general case.

\medskip
\noindent {\bf Case 1:} there exists $C \in \Miss(g)$ such that $\depth(g \circ f, C) = 1$.\\
Then there exist SACs $\alpha_1, \dots, \alpha_n$ satisfying $g \circ f = \alpha_1 \circ \cdots \circ \alpha_{n}$
and $\Miss(\alpha_1)= \{C\}$.
We note that $n(g,C) = n(g\circ f, C) = 1$, where the first equality follows from 
\ref{90f20rdw9019jds}\eqref{70209ru29fyt2suddh29} and the second from \eqref{8234hd1grthdenv},
and where  the assumption $C \in \Miss(g)$ is needed for the first equality.
As $n(g,C) = 1$, there is a unique fundamental point $P$ of $g$ lying on $C$.
Consider the fundamental point $P_1$ of $\alpha_1$;
then \ref{89x21bxnrmv7tr36} implies that $P$ and $P_1$ are fundamental points of $g \circ f$ (lying on $C$);
as $n(g\circ f,C)=1$, $P=P_1$;
so $\alpha_1$ is a SAC with missing curve $C$ and fundamental point $P$.
By \ref{nouveau-dkfjaiuer834134k1j3}\eqref{8734yt9873y4reyw}, there exists $g' \in \BirA$ such that
$g = \alpha_1 \circ g'$. Then 
$\alpha_1 \circ g' \circ f = g \circ f = \alpha_1 \circ \cdots \circ \alpha_{n}$; cancelling $\alpha_1$ yields
$g' \circ f = \alpha_2 \circ \cdots \circ \alpha_{n} \in \Aeul$.
As $n( g' \circ f ) = n-1$, we get $g', f \in \Aeul$ by $(*)$.
Then $g = \alpha_1 \circ g' \in \Aeul$ as well, so we are done in Case~1.

\medskip
\noindent {\bf Case 2:} $n(f)=1$.\\
Note that $f$ is a SAC; let $C$ be its missing curve.
By \eqref{8234hd1grthdenv}, we may consider $D_1 \in \Miss(g \circ f)$ satisfying $\depth(g \circ f, D_1)=1$
and  $n(g \circ f, D_1)=1$.

By Case~1, we may assume that $D_1 \notin \Miss(g)$. 
Then (by \ref{Doifqrqkkdsjf8}) $D_1$ is the closure of $g(C)$; since $C \isom \aff^1$,
we have in fact $g(C)=D_1$ (every dominant morphism $\aff^1 \to C$ is surjective).
Let $P$ be the unique fundamental point of $g\circ f$ on $D_1$ and let $Q \in C$ be the fundamental point of $f$;
then $g(Q) \in D_1$ is a fundamental point of $g\circ f$ by \ref{89x21bxnrmv7tr36}, so $g(Q)=P$.

Since $n(g \circ f, D_1)=1 = n(f,C)$,
\ref{90f20rdw9019jds}\eqref{lwoj837456djd93} implies that $C \cap \exc(g) = \emptyset$.
By \ref{788hkckdssljkhjh25}, $g$ restricts to an isomorphism $\aff^2 \setminus \exc(g) \to \aff^2 \setminus \Gamma_g$,
where $\Gamma_g$ is the union of all missing curves of $g$.
Since $C \subset \aff^2 \setminus \exc(g)$ and $D_1 = g(C)$, if follows that $D_1 \subset \aff^2 \setminus \Gamma_g$.
Since $P \in D_1 \subset \aff^2 \setminus \Gamma_g$ and $\cent(g) \subseteq \Gamma_g$,
we have $P \notin \cent(g)$ and hence $n(g,P)=0$; so \ref{90f20rdw9019jds} gives
$n( g \circ f ,  P ) = n(g,P) + \sum_{P' \in \{Q\}} n(f,P') = 1$
and we have shown
\begin{equation} \label {787832whhswuihp9jxbnmrmidhf}
\text{$D_1 \cap \Gamma_g =\emptyset$ and $n(g\circ f,P)=1$.} 
\end{equation}

Since $\depth(g \circ f, D_1)=1$,
we may choose a factorization $g \circ f = \alpha_1 \circ \cdots \circ \alpha_{n}$ of $g \circ f$ into SACs satisfying
$\Miss(\alpha_1) = \{D_1\}$.
We have $\cent(\alpha_1) = \{P\}$ because the fundamental point of $\alpha_1$ is 
a fundamental point of $g \circ f$ lying on $D_1$.
Write
$\Cont(\alpha_1) = \{E_1\}$, then by \ref{90f20rdw9019jds},
$$
n( \alpha_1 \circ (\alpha_2 \circ \cdots \circ \alpha_n), P)
=  n(\alpha_1,P) + \sum_{P' \in E_1} n( \alpha_2 \circ \cdots \circ \alpha_n , P')
$$
where the left hand side is equal to $n( g \circ f ,  P ) = 1$ by \eqref{787832whhswuihp9jxbnmrmidhf}.
As $n(\alpha_1,P)=1$, we have $n( \alpha_2 \circ \cdots \circ \alpha_n , P') = 0$ for all $P' \in E_1$,
so $\cent( \alpha_2 \circ \cdots \circ \alpha_n ) \cap E_1 = \emptyset$
and in particular $\cent( \alpha_2 ) \cap E_1 = \emptyset$.
It follows that the missing curve $C_2$ of $\alpha_2$ is not equal to $E_1$ (because $\cent(\alpha_2) \subset C_2$). 
So the closure of $\alpha_1(C_2)$ in $\aff^2$ is a curve $D_2$ such that 
$$
D_2 \in \Miss(g \circ f) \setminus \{D_1\} = \Miss(g).
$$
Then $D_2 \subseteq \Gamma_g$, so $D_2 \cap D_1 = \emptyset$ by \eqref{787832whhswuihp9jxbnmrmidhf}.
If $C_2 \cap E_1 \neq \emptyset$ then $\alpha_1(C_2) \cap \alpha_1(E_1) \neq \emptyset$, so $P \in D_2$,
contradicting $D_2 \cap D_1 = \emptyset$; thus 
$$
C_2 \cap E_1 = \emptyset.
$$
This allows us to use \ref{092390hr3tgc}.  By that result, there exist SACs $\alpha_1', \alpha_2'$
such that $\alpha_1 \circ \alpha_2 = \alpha_1' \circ \alpha_2'$ and $\Miss(\alpha_1')=\{D_2\}$.
Since $g \circ f = \alpha_1' \circ \alpha_2' \circ \alpha_3 \circ \cdots \circ \alpha_n$ is a factorization
of $g \circ f$ into SACs satisfying $\Miss(\alpha_1') = \{ D_2 \}$, we have $\depth( g \circ f, D_2 ) = 1$.
Since $D_2 \in \Miss(g)$, Case~1 implies that $f,g \in \Aeul$.

\medskip
\noindent {\bf General case.}
The result is trivial if $n(f)=0$, and follows from Case~2 if $n(f)=1$. 
So we may assume that $n(f)\ge2$.  Consequently, $n(g) \le n-2$.

By \eqref{8234hd1grthdenv}, we may pick $D \in \Miss(g \circ f)$ satisfying $\depth(g \circ f, D)=1$
and  $n(g \circ f, D)=1$.
By Case~1, we may assume that $D \notin \Miss(g)$. 
Then $D$ is the closure of $g(C)$ for some $C \in \Miss(f)$.
We have $1 \le n(f,C) \le  n(g\circ f, D) = 1$, so $n(f,C) = 1$.
Then \ref{nouveau-dkfjaiuer834134k1j3}\eqref{8734yt9873y4reyw} implies that
there exist an SAC $\alpha$ and some $f' \in \BirA$ such that $f = \alpha \circ f'$ and $\Miss(\alpha) = \{C\}$.
On the other hand, the fact that $\depth(g \circ f, D)=1$ allows us to choose 
a factorization $g \circ f = \alpha_1 \circ \cdots \circ \alpha_{n}$ of $g \circ f$ into SACs satisfying
$\Miss(\alpha_1) = \{D\}$. 
We have  $D \in \Miss(g \circ \alpha)$ and
$$
n(g \circ \alpha, D) 
\overset{\ref{90f20rdw9019jds}\eqref{70209ru29fyt2suddh29}}{=}
n( g \circ \alpha \circ f', D) = n(g \circ f, D) = 1.
$$
Let $P$ be the unique fundamental point of $g \circ \alpha$ lying on $D$;
then $P$ is a fundamental point of $g \circ f$ and hence is the unique  fundamental point of $g \circ f$ lying on $D$.
As the fundamental point of $\alpha_1$ is a fundamental point of $g \circ f$ lying on $D$, it follows that
$\alpha_1$ is a SAC with missing curve $D$ and fundamental point $P$.
Then \ref{nouveau-dkfjaiuer834134k1j3}\eqref{8734yt9873y4reyw} implies that
there exists $g' \in \BirA$ satisfying $g \circ \alpha = \alpha_1 \circ g'$.
$$
\xymatrix{
\aff^2 \ar[r]^{f} \ar[rd]_{f'} &  \aff^2 \ar[r]^{g} &  \aff^2  \\
	&  \aff^2 \ar[u]_{\alpha} \ar[r]_{g'}  &  \aff^2  \ar[u]_{\alpha_1}
}
$$
Since $\alpha_1 \circ g' \circ f' = g \circ f = \alpha_1 \circ \cdots \circ \alpha_{n}$,
cancelling $\alpha_1$ gives $g' \circ f' = \alpha_2 \circ \cdots \circ \alpha_{n} \in \Aeul$.
By $(*)$, we obtain $f',g' \in \Aeul$.

Since $f' \in \Aeul$, it follows that $f = \alpha \circ f' \in \Aeul$.

Since $g' \in \Aeul$ we get $g \circ \alpha = \alpha_1 \circ g' \in \Aeul$;
we also have $n(g \circ \alpha) < n$, because $n(g) \le n-2$;
so $g \in \Aeul$ by $(*)$.

So $f,g \in \Aeul$.
\end{proof}


\providecommand{\bysame}{\leavevmode\hbox to3em{\hrulefill}\thinspace}
\providecommand{\MR}{\relax\ifhmode\unskip\space\fi MR }
\providecommand{\MRhref}[2]{%
  \href{http://www.ams.org/mathscinet-getitem?mr=#1}{#2}
}
\providecommand{\href}[2]{#2}


\begin{thebibliography}{10}

\bibitem{A-M:line}
S.~Abhyankar and T.T. Moh, \emph{Embeddings of the line in the plane}, J.\
  reine angew.\ Math. \textbf{276} (1975), 148--166.

\bibitem{BartoCassou:RemPolysTwoVars}
E.~Artal and P.~Cassou-Nogu\`es, \emph{One remark on polynomials in two
  variables}, Pacific J. of Math. \textbf{176} (1996), 297--309.

\bibitem{Cassou-BadFG}
P.~Cassou-Nogu{\`e}s, \emph{Bad field generators}, Affine algebraic geometry,
  Contemp. Math., vol. 369, Amer. Math. Soc., Providence, RI, 2005, pp.~77--83.

\bibitem{CN-Daigle:Lean}
P.\ Cassou-Nogu{\`e}s and D.~Daigle, \emph{Lean factorizations of polynomial
  morphisms}, in preparation.

\bibitem{CassouDaigVGVBFG}
\bysame, \emph{Very good and very bad field generators}, preprint, 2013.

\bibitem{CassouRussell:BirEnd}
P.\ Cassou-Nogu{\`e}s and P.~Russell, \emph{Birational morphisms
  {$\mathbb{C}^2\to\mathbb{C}^2$} and affine ruled surfaces}, Affine algebraic
  geometry, Osaka Univ. Press, Osaka, 2007, pp.~57--105.

\bibitem{Dai:GenRatPols}
D.~Daigle, \emph{Generally rational polynomials in two variables}, preprint,
  2013.

\bibitem{Dai:bir}
\bysame, \emph{Birational endomorphisms of the affine plane}, J. Math. Kyoto
  Univ. \textbf{31} (1991), no.~2, 329--358.

\bibitem{Dai:trees}
\bysame, \emph{Local trees in the theory of affine plane curves}, J. Math.
  Kyoto Univ. \textbf{31} (1991), no.~3, 593--634.

\bibitem{Dai:tri2}
\bysame, \emph{Triangular derivations of {$k[X,Y,Z]$}}, J. Pure Appl. Algebra
  \textbf{214} (2010), 1173--1180.

\bibitem{Gan:Kod}
R.~Ganong, \emph{Kodaira dimension of embeddings of the line in the plane}, J.\
  Math.\ Kyoto U. \textbf{25} (1985), 649--657.

\bibitem{Ganong:Survey}
\bysame, \emph{The pencil of translates of a line in the plane}, Affine
  algebraic geometry, CRM Proc. Lecture Notes, vol.~54, Amer. Math. Soc.,
  Providence, RI, 2011, pp.~57--71.

\bibitem{JanThesis}
C.~J. Jan, \emph{On polynomial generators of $\bk(x,y)$}, Ph.D. thesis, Purdue
  University, 1974.

\bibitem{John_uf-monoids71}
R.~E. Johnson, \emph{Unique factorization monoids and domains},
Proc. Amer. Math. Soc., \textbf{28} (1971), 397--404.

\bibitem{Kal:TwoRems}
S.~Kaliman, \emph{Two remarks on polynomials in two variables}, {P}acific {J}.
  {M}ath. \textbf{154} (1992), 285--295.

\bibitem{MiySugie:GenRatPolys}
M.~Miyanishi and T.~Sugie, \emph{Generically rational polynomials}, Osaka J.\
  Math. \textbf{17} (1980), 339--362.

\bibitem{NeumannNorbury:simple}
W.~D. Neumann and P.~Norbury, \emph{Rational polynomials of simple type},
  Pacific J. Math. \textbf{204} (2002), 177--207.

\bibitem{Nishino_68}
T.~Nishino, \emph{Nouvelles recherches sur les fonctions enti\`eres de
  plusieurs variables complexes. {I}}, J. Math. Kyoto Univ. \textbf{8} (1968),
  49--100.

\bibitem{Nishino_69}
\bysame, \emph{Nouvelles recherches sur les fonctions enti\`eres de plusieurs
  variables complexes. {II}. {F}onctions enti\`eres qui se r\'eduisent \`a
  celles d'une variable}, J. Math. Kyoto Univ. \textbf{9} (1969), 221--274.

\bibitem{Nishino_70}
\bysame, \emph{Nouvelles recherches sur les fonctions enti\`eres de plusiers
  variables complexes. {III}. {S}ur quelques propri\'et\'es topologiques des
  surfaces premi\`eres}, J. Math. Kyoto Univ. \textbf{10} (1970), 245--271.

\bibitem{Rus:FieldGen}
K.P. Russell, \emph{Field generators in two variables}, J. Math. Kyoto Univ.
  \textbf{15} (1975), 555--571.

\bibitem{Rus:fg2}
\bysame, \emph{Good and bad field generators}, J. Math.~Kyoto Univ. \textbf{17}
  (1977), 319--331.

\bibitem{Saito_72}
H.~Sait{\=o}, \emph{Fonctions enti\`eres qui se r\'eduisent \`a certains
  polynomes. {I}}, Osaka J. Math. \textbf{9} (1972), 293--332.

\bibitem{Saito_77}
\bysame, \emph{Fonctions enti\`eres qui se r\'eduisent \`a certains
  polyn\^omes. {II}}, Osaka J. Math. \textbf{14} (1977), 649--674.

\bibitem{Sasao_QuasiSimple2006}
I.\ Sasao, \emph{Generically rational polynomials of quasi-simple type}, J.
  Algebra \textbf{298} (2006), 58--104.

\bibitem{ShpilYu:BirMor}
V.~Shpilrain and J.-T. Yu, \emph{Birational morphisms of the plane}, Proc.
  Amer. Math. Soc. \textbf{132} (2004), no.~9, 2511--2515 (electronic).

\bibitem{Suzuki}
M.~Suzuki, \emph{Propri\'{e}t\'{e}s topologiques des polyn\^omes de deux
  variables complexes, et automorphismes alg\'ebriques de l'espace ${C}^2$}, J.
  Math.\ Soc.\ Japan \textbf{26} (1974), 241--257.

\end{thebibliography}
\end{document}